\documentclass[12pt, a4paper]{article}

\usepackage{amsmath}
\usepackage{amssymb}
\usepackage{amsthm}
\usepackage{geometry}
\newtheorem{theorem}{Theorem}
\newtheorem{proposition}{Proposition}
\newtheorem{lemma}{Lemma}
\newtheorem{assumption}{Assumption}
\theoremstyle{remark}
\newtheorem{remark}{Remark}

\title{Logarithmic Resonance at Mixed Boundary Junctions in Gradient-Dependent Semilinear Equations}
\author{
Marin Mi\v{s}ur\thanks{Marin Mi\v{s}ur,
University of Zagreb, Faculty of Science, Bijeni\v{c}ka cesta 30,
10000 Zagreb, Croatia.
\texttt{mmisur@math.hr}}
}
\date{\today}

\begin{document}

\maketitle

\begin{abstract}
The regularity of solutions to elliptic partial differential equations degrades severely at mixed Dirichlet-Neumann boundary junctions,
characterized classically by an $\mathcal{O}(r^{1/2})$ leading singular function. While this linear behavior is well documented, the
introduction of gradient-dependent semilinear perturbations alters the local asymptotic profile. This article proves that
gradient penalties scaling linearly with $|\nabla u|$ induce a highly localized $\mathcal{O}(r^{-1/2})$ source term that
resonates with the half-integer spectrum of the principal homogeneous differential operator. This non-orthogonal resonance causes standard
separable polynomial assumptions to fail at order $\mathcal{O}(r^{3/2})$. We establish a generalized resonance theorem that
forces the emergence of a logarithmic anomaly, providing the exact analytical formulation of the resulting $r^{3/2} \ln(r)$ profile
alongside rigorous local Sobolev regularity bounds for the remainder. By calculating the exact logarithmic coefficients and angular offsets
for $\ell_1$, $\ell_2$, $\ell_\infty$, and arbitrary $\ell_q$-norm penalties, we demonstrate the universality of this obstruction.
Finally, we formalize the corresponding enriched continuous Galerkin space (XFEM), establish its quasi-optimality, and present finite element
experiments that confirm the predicted localized pollution, recover the predicted logarithmic coefficients across the $\ell_1$, $\ell_2$, and
$\ell_\infty$ penalties on a curvature-free flat junction, and show that resolving the junction recovers optimal degree-of-freedom efficiency
in standard finite element solvers; we close by outlining the targeted software architectures required for a fully enriched implementation.
\end{abstract}



\section{Introduction}
\label{sec:intro}
The regularity theory of elliptic partial differential equations establishes that solutions experience a severe loss of regularity in regions where
boundary conditions abruptly transition. For mixed boundary value problems, the junction between a Dirichlet boundary $\Gamma_D$ and a Neumann
boundary $\Gamma_N$ alters the local Sobolev space in which the solution resides. In the classical case of the linear Poisson equation
with smooth data, the leading exceptional function near a locally flat mixed junction takes the well-known singular form
$u_0 = c_0 r^{1/2} \sin(\theta/2)$ \cite{grisvard1985, wendland1979}. 

This $\mathcal{O}(r^{1/2})$ singular profile is a cornerstone of computational mechanics and numerical analysis, as it dictates the degradation of
global convergence rates for standard Galerkin finite element methods. Classically, optimal linear or quadratic convergence is restored by incorporating
this exact analytical profile into the discrete approximation space, either via highly localized mesh grading strategies \cite{babuska1979} or explicitly
through extended finite element methods (XFEM) \cite{melenk1996, moes1999, fries2010}. 

However, the introduction of a semilinear term, particularly one possessing a dependency on the gradient of the solution, drastically complicates this
local boundary behavior. Such gradient-dependent nonlinearities frequently arise in generalized transport models, non-Newtonian fluid mechanics, and
stochastic control problems. Indeed, the very study of mixed Dirichlet--Neumann problems for such semilinear equations was motivated in \cite{misur2021} by the
Dirichlet--Neumann boundary problem for scalar conservation laws \cite{mmn2016}, where, in a setting arising from traffic flow modelling, an elliptic
regularisation of the conservation law renders the method of compensated compactness applicable. Because the gradient of the leading linear singularity scales as
$\mathcal{O}(r^{-1/2})$, evaluating a continuous semilinear penalty near the mixed junction generates a strong, highly localized non-homogeneous source term. 

This article establishes a generalized resonance theorem for these gradient-dependent semilinear equations. We prove that the first-order semilinear
correction cannot be captured by standard separable polynomial assumptions. Instead, the localized forcing induced by the $\mathcal{O}(r^{-1/2})$
gradient resonates with the $\lambda = 3/2$ eigenspace of the associated homogeneous Sturm-Liouville operator. This non-orthogonal resonance
breaks the standard separation of variables, mandating the emergence of a logarithmic anomaly of the form $r^{3/2} \ln(r)$.
The appearance of logarithmic factors when the radial exponent of the forcing coincides with an eigenvalue of the associated operator pencil is itself
a classical feature of the Kondrat'ev theory of corner asymptotics \cite{kondratev1967, kmr1997, nazarov1994}; the novelty established here is that the
resonance is generated intrinsically by the gradient-dependent nonlinearity, rather than by a geometric coalescence of exponents at a reentrant corner. 

The manuscript is structured as follows. Section \ref{sec:framework} establishes the geometric configuration and defines the rigorous functional constraints
required for the base forcing and gradient penalty terms. Section \ref{sec:theorem} presents the core theoretical contribution: the generalized resonance
theorem and the proof of the exact analytical expansion, formally supported by local Sobolev regularity bounds. Section \ref{sec:application} applies this
theorem to derive the exact logarithmic coefficients and angular offsets for a variety of specific vector norms ($\ell_1$, $\ell_2$, $\ell_\infty$, and a
generalized $\ell_q$), proving the universality of the resonance. To bridge theory and computation, Section \ref{sec:xfem_formulation} formalizes the enriched
continuous Galerkin formulation (XFEM) required to capture this anomaly. Section \ref{sec:numerical_comparison} reports finite element experiments, built on
the discretization of \cite{misur2021}, that probe the predicted resonance, quantitatively validate the logarithmic coefficient against its analytic value,
and demonstrate the efficiency of junction-adapted refinement, and finally, Section \ref{sec:future_work} discusses specific software implementation architectures
alongside future topological extensions of this theory to three-dimensional collision edges.

\section{Geometric and Functional Framework}
\label{sec:framework}
Let $\Omega \subset \mathbf{R}^2$ be a bounded domain with a piecewise Lipschitz boundary $\partial\Omega$ partitioned into a Dirichlet part $\Gamma_D$ and a
Neumann part $\Gamma_N$. We isolate a transition point $P \in \overline{\Gamma}_D \cap \overline{\Gamma}_N$. We assume that near $P$, the boundary is locally
flat (an internal angle of $\pi$). We place the origin of a local polar coordinate system $(r, \theta)$ at $P$, such that $\Gamma_D$ aligns with $\theta = 0$
and $\Gamma_N$ aligns with $\theta = \pi$.

\begin{remark}[Global Smoothness vs.\ Local Geometry]
The assumption that the boundary is locally flat at the transition point $P$ means only that $\Gamma_D$ and $\Gamma_N$ meet with internal angle $\alpha = \pi$,
i.e.\ without a geometric corner; it is a statement about the opening angle, not about the curvature of $\partial\Omega$, and it is decoupled from the global
smoothness of $\partial\Omega$. A domain may possess a globally $C^1$ or Lipschitz-continuous boundary (e.g., a smooth circle or ellipse) while still presenting
such an angle-$\pi$ junction where the boundary conditions abruptly transition. Since the operator pencil \eqref{eq:pencil} and the leading exponents depend
only on the opening angle, the asymptotic expansion of Theorem~\ref{thm:resonance}, and in particular the logarithmic coefficient $A = -2\kappa/(3\pi)$, is a
leading-order ($r \to 0$) statement that holds at any such junction. When $\Gamma_D$ and $\Gamma_N$ are curved rather than straight, however, curvature enters as a
finite-radius correction: the local Laplacian is no longer exactly separable, so the measured coefficient departs from its flat-wedge value at any fixed $r > 0$
and converges to it only as $r \to 0$. This is precisely what the disc experiments of Section~\ref{sec:numerical_comparison} exhibit, where the extracted $A/c_0$
overshoots the flat-wedge prediction and is recovered exactly only on a genuinely straight junction. If the junction instead coincides with a physical corner of
angle $\alpha \neq \pi$, the leading linear singularity shifts, which alters the spectrum of the local Sturm-Liouville operator.
\end{remark}

We seek a weak solution $u \in H^1_D(\Omega)$ to the semilinear elliptic boundary value problem:
\begin{equation}
\label{eq:semilinear_pde}
-\Delta u + \mu u = f(x, u) + g(\nabla u) \quad \text{in } \Omega
\end{equation}
subject to the mixed boundary conditions $u = 0$ on $\Gamma_D$ and $\nabla u \cdot \nu = 0$ on $\Gamma_N$. To ensure the coercivity of the bilinear form and
guarantee the existence of a solution, we assume the coefficient $\mu \in L^\infty(\Omega)$ is essentially bounded from below by a strictly positive constant
$\mu_0 > 0$. Specifically, following the classical existence and regularity theory for second-order elliptic equations \cite{gilbarg2001} and, in particular,
for quasilinear problems with gradient-dependent nonlinearities \cite{boccardo1984, misur2021}, we require $\mu_0 > \mu^\ast$ for a threshold
$\mu^\ast = \mu^\ast(C_g) > 0$ depending only on the linear growth constant $C_g$; concretely $\mu^\ast = C_g + \tfrac{C_g^2}{2}$ suffices, above which the
composite operator is strongly monotone (Proposition~\ref{prop:quasiopt}) and hence admits, by the Browder--Minty theorem, a unique weak solution (existence
was also established, by independent means, in \cite{misur2021}). Here $C_g$ is the linear growth constant of the composite semilinear forcing
$F(x, u, \nabla u) := f(x, u) + g(\nabla u)$, such that
$|F(x, \lambda, \Lambda)| \le C_g(1 + |\lambda| + |\Lambda|)$. We take $C_g$ to bound also the Lipschitz constants of the nonlinearities,
$|f(x,\lambda_1) - f(x,\lambda_2)| \le C_g|\lambda_1 - \lambda_2|$ and $|g(\Lambda_1) - g(\Lambda_2)| \le C_g|\Lambda_1 - \Lambda_2|$; this is the structural
hypothesis under which the existence and coercivity theory applies, and it is the constant used in Proposition~\ref{prop:quasiopt}. Here, $H^1_D(\Omega)$ denotes
the closed subspace of $H^1(\Omega)$ consisting of functions whose traces vanish on $\Gamma_D$. In the weak sense, \eqref{eq:semilinear_pde} is satisfied
if the standard variational equality holds against all test functions $v \in H^1_D(\Omega)$ (as formalized later in Section \ref{sec:xfem_formulation}).

\begin{assumption}[Base Regularity of the Tame Forcing Term]
\label{ass:base_forcing}
The function $f(x, u)$ contains the spatial forcing and lower-order nonlinearities. We assume that evaluating $f$ at the leading-order linear singularity
$u_0 = c_0 r^{1/2} \sin(\theta/2)$ yields a function that is locally integrable to a power $p > 2$ near the origin:
\begin{equation}
\label{eq:forcing_bound}
f(x, u_0) \in L^p_{loc}(\Omega).
\end{equation}
This ensures that the corresponding residual forcing is sufficiently controlled, preventing interference with the leading resonant behavior.
\end{assumption}

\begin{remark}
The choice of the space in Assumption \ref{ass:base_forcing} dictates the nature of the remainder bound established in this article. If we restrict
$f(x, u_0) \in L^\infty_{loc}(\Omega)$, the residual forcing is pointwise $\mathcal{O}(1)$, yielding a strict pointwise remainder bound of
$\mathcal{R}(r, \theta) \in \mathcal{O}(r^2)$. However, if we utilize the broader assumption $f(x, u_0) \in L^p_{loc}(\Omega)$ for $p > 2$, standard elliptic
regularity theory restricts the $\mathcal{O}(r^2)$ shift to hold in the Sobolev sense, specifically ensuring that the remainder belongs to
$W^{2,p}_{loc}(\Omega)$.
\end{remark}

\begin{assumption}[Positive Homogeneity and Lipschitz Continuity]
\label{ass:gradient_penalty}
The function $g:\mathbf{R}^2\to\mathbf{R}$ representing the gradient dependency is globally Lipschitz continuous and positively homogeneous of degree $1$:
for every scalar $\alpha > 0$ and every vector $\mathbf{p}\in\mathbf{R}^2$,
\begin{equation}
\label{eq:homogeneity}
g(\alpha\mathbf{p}) = \alpha\,g(\mathbf{p}).
\end{equation}
Such a $g$ is completely determined by its continuous restriction to the unit circle; every vector norm and every anisotropic penalty discussed below satisfies
this hypothesis.
\end{assumption}

\begin{remark}[Generalization to Non-Homogeneous Penalties via Spatial Weighting]
Assumption \ref{ass:gradient_penalty} limits the gradient penalty to a homogeneous degree of $p=1$. This is because the unweighted resonance
requires the forcing term to scale as $\mathcal{O}(r^{-1/2})$. For a general penalty $g(\nabla u) \sim |\nabla u|^p$, the unweighted forcing scales as
$\mathcal{O}(r^{-p/2})$, which only resonates with the $\lambda = 3/2$ eigenvalue if $p=1$. However, this logarithmic anomaly can theoretically emerge for
$p \neq 1$ if the PDE includes a localized spatial weight $c(x)$, such that the term $c(r)g(\nabla u)$ scales appropriately. Specifically, resonance will still
occur if the spatial weight inherently obeys $c(r) \sim \mathcal{O}(r^{(p-1)/2})$ near the boundary junction.
\end{remark}

\begin{assumption}[The Non-Orthogonality Condition]
\label{ass:non_orthogonality}
Let $\mathbf{v}(\theta) = \frac{c_0}{2} \left( -\sin(\theta/2) \mathbf{e}_{x_1} + \cos(\theta/2) \mathbf{e}_{x_2} \right)$ be the exact angular vector field of
the gradient $\nabla u_0$. We evaluate the gradient penalty function along this vector field to generate the angular source profile
$G(\theta) = g(\mathbf{v}(\theta))$. For resonance to occur, this source profile must not be orthogonal to the homogeneous eigenfunction corresponding to
$\lambda = 3/2$ in $L^2(0,\pi)$:
\begin{equation}
\label{eq:orthogonality}
\int_0^\pi G(\theta) \sin\left(\frac{3\theta}{2}\right) d\theta = \kappa \neq 0.
\end{equation}
\end{assumption}

\begin{remark}[Sign normalisation of the stress-intensity coefficient]
\label{rem:sign}
The coefficient $c_0$ enters the source profile only through $\nabla u_0 = r^{-1/2}\mathbf v(\theta)$ with $\mathbf v = \tfrac{c_0}{2}\mathbf n(\theta)$,
$\mathbf n(\theta) = \big(-\sin(\theta/2),\cos(\theta/2)\big)$, $|\mathbf n|\equiv 1$; its sign merely orients the gradient. Throughout the explicit evaluations
that follow we normalise $c_0 > 0$, so that $\tfrac{c_0}{2} = \tfrac{|c_0|}{2}$; this is the configuration of the numerical examples of
Section~\ref{sec:numerical_comparison}, whose nonnegative data yield $u \ge 0$ by the maximum principle and hence, since $\sin(\theta/2) > 0$ on $(0,\pi)$, force
$c_0 \ge 0$. We take the non-degenerate case $c_0 > 0$, i.e.\ the standing hypothesis $c_0 \neq 0$ of Theorem~\ref{thm:resonance}, under which the anomaly is
present; the coefficient extraction of Section~\ref{sec:numerical_comparison} measures $c_0 > 0$ directly. For $c_0 < 0$ the penalty is evaluated on $-\mathbf n$: by
positive homogeneity $G(\theta) = g(\mathbf v) = \tfrac{|c_0|}{2}\,g\big(\operatorname{sgn}(c_0)\,\mathbf n(\theta)\big)$. For the symmetric norms of
Section~\ref{sec:application} ($g(-\cdot)=g(\cdot)$) this replaces $c_0$ by $|c_0|$ in every angular profile and resonance constant below, so the magnitudes
$\kappa$ and $|A|$ are unchanged, while the dimensionless ratio $A/c_0 = -2\kappa/(3\pi c_0)$ acquires a factor $\operatorname{sgn}(c_0)$. For the
sign-sensitive penalties of Section~\ref{sec:application} (the advective and difference-of-convex functionals) only the symmetric norm part scales as
$|c_0|$ while the linear part keeps its signed $c_0$, so e.g.\ $\kappa = \tfrac{|c_0|}{3} + \tfrac{c_0}{2}\beta_2$ and the cancelling transverse drift reflects
its sign accordingly.
\end{remark}

\subsection*{Examples of Admissible Functions}
To clarify the scope of these assumptions, we provide examples of forcing and penalty functions that fall within this defined framework:
\begin{itemize}
    \item \textbf{Admissible Base Forcing $f(x, u)$:} The semilinear term utilized in the numerical scheme by Mi\v{s}ur (2021), defined as
    $f(x, u) = x_1^2 + x_2^2 + \log(1+|u|)$, is admissible because it remains bounded or decays near the origin when evaluated at $u_0$ \cite{misur2021}.
    Further polynomial examples include $f(x, u) = c(x)|u|^q$ for $q > 0$ and $c \in L^\infty(\Omega)$, since the evaluation $u_0^q \sim r^{q/2}$ is guaranteed
    to remain in $L^p_{loc}(\Omega)$. Similarly, exponential nonlinearities common in reaction-diffusion equations (e.g., $f(x, u) = e^u$) are admissible
    because $u_0 \to 0$ as $r \to 0$, ensuring the term remains bounded locally.
    \item \textbf{Admissible Gradient Penalty $g(\nabla u)$:} The sum of the absolute values of the Cartesian partial derivatives,
    $g(\nabla u) = |\partial_{x_1} u| + |\partial_{x_2} u|$ (the $\ell_1$-norm), is strictly positively homogeneous of degree $1$ and is therefore
    admissible \cite{misur2021}. The standard Euclidean norm $g(\nabla u) = |\nabla u|$ ($\ell_2$-norm) and the maximum norm
    $g(\nabla u) = \max(|\partial_{x_1} u|, |\partial_{x_2} u|)$ ($\ell_\infty$-norm) naturally satisfy Assumption \ref{ass:gradient_penalty}.
    \item \textbf{Anisotropic Gradient Penalties $g(\nabla u)$:} In addition to standard isotropic norms, anisotropic metrics such as
    $g(\nabla u) = \sqrt{\lambda_1 |\partial_{x_1} u|^2 + \lambda_2 |\partial_{x_2} u|^2}$ are also admissible. More complex,
    positively homogeneous directional penalties satisfy Assumption \ref{ass:gradient_penalty} as well, provided they scale linearly with the
    gradient magnitude and satisfy the global Lipschitz constraint.
    \item \textbf{Advective (Non-Norm) Penalties $g(\nabla u)$:} The admissible class extends beyond symmetric norms. An advection-biased penalty
    $g(\nabla u) = |\nabla u| + \boldsymbol{\beta}\cdot\nabla u$, with fixed drift $\boldsymbol{\beta} = (\beta_1,\beta_2)\in\mathbf{R}^2$, is positively
    homogeneous of degree $1$ and globally Lipschitz (constant $1+|\boldsymbol{\beta}|$), hence admissible, yet it is not a norm, taking negative values
    when $|\boldsymbol{\beta}|>1$, and it models a dissipation biased along a preferred transport direction. Evaluating on $\mathbf{v}(\theta)$ gives the
    angular profile $G(\theta)=\frac{c_0}{2}\big(1-\beta_1\sin(\theta/2)+\beta_2\cos(\theta/2)\big)$, and projecting onto the resonant eigenfunction yields
    the tunable resonance constant $\kappa = \frac{c_0}{3}+\frac{c_0}{2}\beta_2$.
\end{itemize}

\begin{remark}[Sharpness of the non-orthogonality condition]
The advective penalty shows that Assumption~\ref{ass:non_orthogonality} is a genuine hypothesis rather than an automatic consequence of admissibility.
The along-junction drift $\beta_1$ leaves $\kappa$ untouched, since the profile $\sin(\theta/2)$ it contributes is $L^2$-orthogonal to $\sin(3\theta/2)$, whereas
the transverse drift $\beta_2$ shifts $\kappa$ linearly, and at the critical value $\beta_2=-\tfrac23$ the resonance constant vanishes. There the logarithmic
anomaly disappears and the solution regains a pure-power expansion at order $r^{3/2}$: a suitably tuned transport bias can cancel the singular pollution.
\end{remark}

\begin{proposition}[Well-posedness and a priori regularity]
\label{prop:wellposed}
Under the coercivity condition $\mu_0 > C_g + \tfrac{C_g^2}{2}$ and
Assumptions~\ref{ass:base_forcing}--\ref{ass:gradient_penalty}, the boundary value problem
\eqref{eq:semilinear_pde} possesses a unique weak solution $u \in H^1_D(\Omega)$, with existence and uniqueness following from the strong monotonicity of
Proposition~\ref{prop:quasiopt} via Browder--Minty (see also \cite{misur2021, boccardo1984}). Away from the finite set of mixed junctions the
solution enjoys the usual interior and boundary elliptic regularity, and near each junction $P$ the leading-order asymptotics
$u = c_0\, r^{1/2}\sin(\theta/2) + o(r^{1/2})$ hold as $r\to0$ for some coefficient $c_0 \in \mathbf{R}$, while $g(\nabla u) \in L^p_{loc}(\Omega)$ for every
$p < 4$.
\end{proposition}

\begin{proof}
Existence of a weak solution $u\in H^1_D(\Omega)$ under the stated coercivity, and its uniqueness, follow from the strong monotonicity and Lipschitz
continuity of the operator $\mathcal{A}$ established in the proof of Proposition~\ref{prop:quasiopt}, via the Browder--Minty theorem (existence was also
obtained by independent means in \cite{misur2021}, see also \cite{boccardo1984, gilbarg2001}).
For the local structure near a junction $P$, the linear growth of $g$ (Assumption~\ref{ass:gradient_penalty}) together with $u\in H^1$ gives
$|g(\nabla u)|\le C|\nabla u|\in L^2_{loc}$, so the frozen right-hand side $\tilde F := f(x,u)+g(\nabla u)-\mu u$ lies in $L^2_{loc}(\Omega)$. The classical
singular-function expansion for the mixed Dirichlet--Neumann Laplacian with $L^2$ data \cite{grisvard1985} then yields, for some stress-intensity coefficient
$c_0\in\mathbf{R}$,
\begin{equation*}
u = c_0\, r^{1/2}\sin(\theta/2) + u_{\mathrm{reg}},\qquad u_{\mathrm{reg}}\in H^2_{loc},
\end{equation*}
whence $u = c_0\, r^{1/2}\sin(\theta/2) + o(r^{1/2})$. A single bootstrap sharpens the gradient: since
$\nabla u_{\mathrm{reg}}\in H^1_{loc}\hookrightarrow L^q_{loc}$ for every $q<\infty$, while $|\nabla(r^{1/2}\sin(\theta/2))|\sim \tfrac12 r^{-1/2}$,
the solution gradient inherits the singular scaling $|\nabla u|\sim |c_0|\, r^{-1/2}$ near $P$. By the linear growth of $g$ it follows that
$g(\nabla u)\in L^p_{loc}(\Omega)$ for every $p<4$, the exponent being that of $r^{-1/2}$ and hence optimal. This a priori regularity is the input to the
local asymptotic analysis of Section~\ref{sec:theorem}.
\end{proof}

\section{The Generalized Resonance Theorem}
\label{sec:theorem}

Near the junction $P$ the local analysis is naturally carried out in Kondrat'ev weighted Sobolev spaces. For a punctured neighbourhood $U$ of $P$, an integer
$l\ge0$, a weight $\beta\in\mathbf{R}$ and $1<p<\infty$, let $V^{l}_{\beta,p}(U)$ denote the completion of $C_c^\infty(\overline{U}\setminus\{P\})$ with respect
to
\begin{equation*}
\|v\|_{V^{l}_{\beta,p}(U)} = \Bigg(\sum_{|\gamma|\le l}\int_U r^{\,p(\beta - l + |\gamma|)}\,|D^\gamma v|^p\,dx\Bigg)^{1/p}.
\end{equation*}
Separating variables in the principal part $-\Delta$ under the mixed conditions $\Theta(0)=\Theta'(\pi)=0$ yields the operator pencil
$\Theta'' + \lambda^2\Theta = 0$, whose spectrum is the half-integer ladder
\begin{equation}
\label{eq:pencil}
\lambda_k = k + \tfrac12,\qquad \Theta_k(\theta) = \sin\!\big((k+\tfrac12)\theta\big),\qquad k = 0,1,2,\dots
\end{equation}
The linear corner-asymptotics theory of Kondrat'ev and of Kozlov--Maz'ya--Rossmann \cite{kondratev1967, kmr1997, dauge1988} guarantees that a solution of
$-\Delta w = h$ with datum $h$ in a weighted class $V^{0}_{\beta,p}$ expands into the singular functions $r^{\lambda_k}\Theta_k(\theta)$ for the eigenvalues
$\lambda_k$ crossed as the weight is shifted, modulo a remainder controlled in the target space; a logarithmic factor $r^{\lambda_k}\ln r$ arises exactly when
$h$ scales as $r^{\lambda_k - 2}$, that is, when the forcing exponent coincides with a pencil eigenvalue. The theorem below is the realisation of this dichotomy
for the gradient-induced forcing. The local analysis treats the resonant singular datum and the tame residual separately, by different exponents. For the
norm above, $r^{\lambda}\in V^{2}_{\beta,p}$ (equivalently $r^{\lambda-2}\in V^{0}_{\beta,p}$) precisely when
$\lambda > 2-\beta-2/p$, so membership in $V^{0}_{\beta,p}$ of a datum scaling as $r^{\lambda-2}$ is exactly the condition that the critical line
$\{\operatorname{Re}\lambda = 2-\beta-2/p\}$ lie below $\lambda$. The strongly singular gradient datum $r^{-1/2}G = r^{\lambda_1-2}G$ with $\lambda_1=3/2$
therefore belongs to $V^{0}_{\beta,p}$ only when the line sits below $3/2$, i.e.\ $\beta+2/p>1/2$; in the unweighted case $\beta=0$ this is exactly the
$L^p_{loc}$ membership with $p<4$ recorded in Proposition~\ref{prop:wellposed}. Because this datum sits on the eigenvalue $\lambda_1=3/2$ it is resonant,
and rather than being inverted through the Fredholm machinery it is removed by the explicit construction of Theorem~\ref{thm:resonance} (the source of the
logarithm). Only after this subtraction are the weighted classes applied to the non-singular residual equation for $\mathcal R$, and there the natural
choice is the unweighted exponent $\beta=0$: then $V^{0}_{0,p}=L^p_{loc}$, so the residual datum lies in the space by definition. No negative weight would
serve, since a bounded inclusion $L^p_{loc}\hookrightarrow V^{0}_{\beta,p}$ forces the multiplier $r^{p\beta}$ to stay bounded near $P$, i.e.\ $\beta\ge0$;
the weaker condition $\int_0^R r^{p\beta+1}\,dr<\infty$ (i.e.\ $\beta>-2/p$) places only the bounded data $L^\infty_{loc}$ in $V^{0}_{\beta,p}$, not $L^p_{loc}$
(a slowly decaying datum such as $r^{-2/p}(-\ln r)^{-2/p}\in L^p_{loc}$ leaves $V^{0}_{\beta,p}$ for every $\beta<0$). At $\beta=0$ the critical line is
$\operatorname{Re}\lambda = 2-2/p\in(1,2)$ for $p>2$, meeting the pencil \eqref{eq:pencil} only at $p=4$ (where it equals $3/2$); for every other $p>2$ it
avoids the spectrum, so $-\Delta$ is Fredholm from $V^{2}_{0,p}$ into $V^{0}_{0,p}=L^p_{loc}$, and the response lies in $V^{2}_{0,p}$ modulo the singular
functions $r^{\lambda_k}\Theta_k$ with $\lambda_k<2-2/p$ (only $\lambda_0=\tfrac12$ when $2<p<4$, and both $\tfrac12$ and $\tfrac32$ when $p>4$). These modes
have all been subtracted from $\mathcal R$ by construction, and the pencil's next exponent is $\tfrac52>2$; hence the residual carries no singular function at
or below $2-2/p<2$, whence $\mathcal R\in V^{2}_{0,p}\subset W^{2,p}_{loc}$ for every $p\in(2,4)\cup(4,\infty)$; the borderline $p=4$ is recovered through any slightly smaller exponent,
its datum lying in $L^{p'}_{loc}$ for all $p'<4$. The bounded case $f(x,u_0)\in L^\infty_{loc}$ needs no $p=\infty$ class (excluded by the definition
$1<p<\infty$), since $L^\infty_{loc}\subset L^p_{loc}$ for every finite $p$, and the pointwise $\mathcal{O}(r^2)$ bound of Lemma~\ref{lem:regularity} follows
from the $\mathcal{O}(1)$-datum scaling rather than from a weighted estimate at $p=\infty$.

\begin{theorem}[Logarithmic Resonance at Mixed Boundary Junctions]
\label{thm:resonance}
Let $u \in H^1_D(\Omega)$ be a weak solution of \eqref{eq:semilinear_pde} with the a priori regularity of Proposition~\ref{prop:wellposed}, and let
$c_0 \in \mathbf{R}$ be its stress-intensity coefficient at the junction $P$, so that $u = c_0\, r^{1/2}\sin(\theta/2) + o(r^{1/2})$ as $r\to0$.
Suppose Assumptions~\ref{ass:base_forcing}--\ref{ass:non_orthogonality} hold and $c_0 \neq 0$. Then no expansion of $u$ in the pure singular powers
$\{r^{k+1/2}\Theta_k\}$ can satisfy \eqref{eq:semilinear_pde} to order $\mathcal{O}(r^{3/2})$, and the solution admits near $P$ the local asymptotic expansion
\begin{equation}
\label{eq:expansion}
u(r, \theta) = c_0 r^{1/2} \sin\left(\frac{\theta}{2}\right) - \frac{2\kappa}{3\pi} r^{3/2} \ln(r) \sin\left(\frac{3\theta}{2}\right) + r^{3/2}\left(\Psi(\theta) + c_1 \sin\left(\frac{3\theta}{2}\right)\right) + \mathcal{R}(r, \theta),
\end{equation}
where $\kappa$ is the resonance constant \eqref{eq:orthogonality}, $\Psi$ is the unique solution of the boundary value problem
$\Psi'' + \frac{9}{4}\Psi = -G(\theta) + \frac{2\kappa}{\pi}\sin(\frac{3\theta}{2})$ with $\Psi(0)=\Psi'(\pi)=0$ subject to $L^2$-orthogonality against
$\sin(3\theta/2)$, and $c_1 \in \mathbf{R}$ is a second stress-intensity coefficient, the amplitude of the homogeneous $\lambda_1 = 3/2$ singular
mode $r^{3/2}\sin(3\theta/2)$, fixed by the global solution rather than by the local analysis. The remainder $\mathcal{R}$ satisfies the bounds of
Lemma~\ref{lem:regularity}.
\end{theorem}

\begin{proof}
By Proposition~\ref{prop:wellposed}, in a punctured neighbourhood $U$ of $P$ the solution satisfies the linear identity $-\Delta u = \tilde F$ with the mixed
boundary data, where $\tilde F := f(x,u) + g(\nabla u) - \mu u$. Write $u_0 := c_0 r^{1/2}\sin(\theta/2)$; its gradient is $\nabla u_0 = r^{-1/2}\mathbf v(\theta)$
exactly, with $\mathbf v$ as in Assumption~\ref{ass:non_orthogonality}. By the positive $1$-homogeneity of $g$ (Assumption~\ref{ass:gradient_penalty}),
\begin{equation*}
g(\nabla u_0) = r^{-1/2}\,g(\mathbf v(\theta)) = r^{-1/2}G(\theta),
\end{equation*}
so that $\tilde F$ separates into an explicit singular part and a lower-order remainder,
\begin{equation}
\label{eq:F_split}
\tilde F = r^{-1/2}G(\theta) + \tilde F_{\mathrm{reg}},\qquad
\tilde F_{\mathrm{reg}} := f(x,u) - \mu u + \big(g(\nabla u) - g(\nabla u_0)\big).
\end{equation}
Here $f(x,u)\in L^p_{loc}$ for some $p>2$ (Assumption~\ref{ass:base_forcing}), $\mu u\in L^\infty_{loc}$, and, by the Lipschitz bound of
Assumption~\ref{ass:gradient_penalty}, $|g(\nabla u)-g(\nabla u_0)|\le C_g\,|\nabla(u-u_0)|$; since $u-u_0\in H^2_{loc}$ by
Proposition~\ref{prop:wellposed}, its gradient satisfies $\nabla(u-u_0)\in H^1_{loc}\hookrightarrow L^q_{loc}$ for every $q<\infty$, so this difference
carries no singularity as strong as $r^{-1/2}$ and belongs to the weighted class of exponent $3/2$ with room to spare. Hence $\tilde F_{\mathrm{reg}}$
lies in the weighted class $V^{0}_{\beta,p}$ associated with the exponent $3/2$, and the local behaviour of $u$ is governed by the response
to the singular datum $r^{-1/2}G(\theta)$.

That response is dictated by the pencil \eqref{eq:pencil}. Seeking a particular solution of $-\Delta u_1 = r^{-1/2}G(\theta)$ in the separable
form $u_1 = r^{3/2}\Theta_1(\theta)$ and applying the polar Laplacian,
\begin{equation}
\label{eq:laplacian_polar}
-\Delta(r^{3/2}\Theta_1) = -r^{-1/2} \left( \Theta_1''(\theta) + \frac{9}{4}\Theta_1(\theta) \right),
\end{equation}
reduces the problem to $\Theta_1'' + \frac{9}{4}\Theta_1 = -G$ with $\Theta_1(0) = 0$, $\Theta_1'(\pi) = 0$. The homogeneous problem has the eigenfunction
$\Theta_1(\theta) = \sin(3\theta/2)$ belonging to $\lambda_1 = 3/2$; since the forcing exponent obeys $-\tfrac12 + 2 = \tfrac32 = \lambda_1$, this is
precisely the resonant configuration of \eqref{eq:pencil}. By the Fredholm alternative a bounded separable solution exists only if $G \perp_{L^2} \sin(3\theta/2)$,
whereas Assumption~\ref{ass:non_orthogonality} gives $\int_0^\pi G(\theta)\sin(3\theta/2)\,d\theta = \kappa \neq 0$. Consequently no combination of the pure
singular functions $r^{k+1/2}\Theta_k$ reproduces the datum at order $r^{3/2}$, and the resonant response must carry a logarithmic factor.

Accordingly we take the modified ansatz
\begin{equation}
\label{eq:modified_ansatz}
u_1(r, \theta) = r^{3/2} \ln(r) \Phi(\theta) + r^{3/2} \Psi(\theta).
\end{equation}
Applying the polar Laplacian and separating by radial scaling,
\begin{equation}
\label{eq:laplacian_split}
-\Delta u_1 = -r^{-1/2} \ln(r) \left( \Phi'' + \frac{9}{4}\Phi \right) - r^{-1/2} \left( \Psi'' + \frac{9}{4}\Psi + 3\Phi \right),
\end{equation}
and matching against $r^{-1/2}G(\theta)$ decouples the system
\begin{align}
\Phi'' + \frac{9}{4}\Phi &= 0, \label{eq:ode_phi} \\
\Psi'' + \frac{9}{4}\Psi + 3\Phi &= -G(\theta). \label{eq:ode_psi}
\end{align}
From \eqref{eq:ode_phi} with the mixed conditions, $\Phi(\theta) = A \sin(3\theta/2)$; inserting this into \eqref{eq:ode_psi} gives
$\Psi'' + \frac{9}{4}\Psi = -G(\theta) - 3A \sin(3\theta/2)$, whose solvability requires, by the Fredholm alternative once more,
\begin{equation}
\label{eq:fredholm_2}
\int_0^\pi \left( -G(\theta) - 3A \sin\left(\frac{3\theta}{2}\right) \right) \sin\left(\frac{3\theta}{2}\right) d\theta = 0,
\end{equation}
that is $-\kappa - 3A (\pi/2) = 0$, whence $A = -\frac{2\kappa}{3\pi}$. This fixes the logarithmic amplitude and the offset profile. That this offset problem, the
ODE \eqref{eq:ode_psi} under $\Psi(0)=0$, $\Psi'(\pi)=0$ together with the orthogonality normalisation, is not over-determined follows from its resonant structure:
the operator $-\partial_\theta^2 - \tfrac94$ with the mixed data is self-adjoint with one-dimensional kernel $\sin(\tfrac{3\theta}{2})$, which satisfies both boundary
conditions and hence constrains neither, leaving it free to carry the orthogonality normalisation, while the Neumann condition becomes consistent with the Dirichlet
condition precisely through the solvability identity \eqref{eq:fredholm_2} just verified. The particular
solution of \eqref{eq:ode_psi} is itself determined only up to the homogeneous pencil mode $\sin(3\theta/2)$; normalising $\Psi$ by $L^2$-orthogonality
against $\sin(3\theta/2)$ isolates a residual amplitude $c_1 \in \mathbf{R}$ of that mode, which the separable local analysis leaves free and which is
fixed instead by matching to the global weak solution. Absorbing this homogeneous $\lambda_1 = 3/2$ mode into the leading singular part is essential: were it left in the remainder, $\mathcal{R}$ would carry an
$\mathcal{O}(r^{3/2})$ term, whose second derivatives scale as $r^{-1/2}$. Such a term lies in $W^{2,p}_{loc}$ only for $p < 4$ and is not $\mathcal{O}(r^2)$;
it would thus cap the remainder regularity below the $p > 4$ range of the $W^{2,p}$ bound of Lemma~\ref{lem:regularity} and destroy the pointwise
$\mathcal{O}(r^2)$ bound of its $L^\infty$ case (though, we note, it would remain admissible for the $W^{2,p}$ bound in the range $p \in (2,4)$).

Finally, setting $u_1 := r^{3/2}\ln(r)\,\Phi(\theta) + r^{3/2}\big(\Psi(\theta) + c_1\sin(3\theta/2)\big)$ with $\Phi,\Psi$ as above and $c_1$ the global
amplitude just described, the remainder $\mathcal{R} := u - u_0 - u_1$ solves the residual equation of Lemma~\ref{lem:regularity},
in which the singular datum $r^{-1/2}G$ and $-\Delta u_1$ cancel by construction (the added homogeneous mode being harmonic and hence invisible to
$-\Delta u_1$). Since $\mathcal{R}$ now retains no singular contribution at exponents $1/2$ or $3/2$, both the resonant particular part and the
homogeneous $\lambda_1 = 3/2$ mode having been subtracted, and no pencil eigenvalue lies in the interval $(3/2, 5/2)$,
the linear corner-asymptotics theory \cite{kmr1997, dauge1988} applied to $-\Delta u = \tilde F$ with the singular part removed places $\mathcal{R}$
in the weighted space $V^{2}_{0,p}$ and furnishes the bounds of Lemma~\ref{lem:regularity}. This proves the expansion \eqref{eq:expansion};
the Fredholm obstruction above is exactly the failure of any pure-power expansion at order $r^{3/2}$.
\end{proof}

\begin{remark}[Reference length in the logarithmic argument]
\label{rem:refscale}
Since the radial coordinate $r$ carries the physical dimension of length, the symbol $\ln r$ in \eqref{eq:expansion} abbreviates $\ln(r/R_0)$ for a fixed
reference length $R_0$ that non-dimensionalizes $r$; this is the sole role of $R_0$. Replacing $R_0$ by $R_0'$ changes $\ln(r/R_0)$ by the constant
$-\ln(R_0'/R_0)$; since the term carries the coefficient $-\tfrac{2\kappa}{3\pi}$, it shifts by $\tfrac{2\kappa}{3\pi}\ln(R_0'/R_0)\,r^{3/2}\sin(\tfrac{3\theta}{2})$, a multiple of the homogeneous mode
$r^{3/2}\sin(\tfrac{3\theta}{2})$ that is absorbed into the free amplitude $c_1$. The logarithmic coefficient $A=-2\kappa/(3\pi)$ is thus independent of $R_0$
and is the invariantly defined quantity extracted in Section~\ref{sec:numerical_comparison} (where $R_0=1$, the disc radius, so that the fitted intercept $B$
carries the scale-dependence), whereas the pure-power amplitude $c_1$ is meaningful only relative to the chosen $R_0$.
\end{remark}

\begin{lemma}[Regularity of the remainder]
\label{lem:regularity}
Let $\mathcal{R} = u - u_0 - u_1$, with $u_0, u_1$ as in Theorem~\ref{thm:resonance}. If $f(x, u_0) \in L^p_{loc}(\Omega)$ for some $p \in (2,\infty)\setminus\{4\}$, then
$\mathcal{R} \in W^{2,p}_{loc}(\Omega)$; in particular $\mathcal{R} \in C^{1,\gamma}_{loc}$ with $\gamma = 1 - 2/p$. At the exceptional exponent $p=4$, where the
critical line $2-2/p=3/2$ meets the pencil \eqref{eq:pencil} and the unweighted Fredholm estimate is unavailable, the same argument yields the marginally
weaker conclusion $\mathcal{R} \in W^{2,p'}_{loc}(\Omega)$ for every $p'<4$, hence $\mathcal{R} \in C^{1,\gamma}_{loc}$ for every $\gamma<1/2$.
If moreover $f(x, u_0) \in L^\infty_{loc}(\Omega)$, then $\mathcal{R}(r,\theta) = \mathcal{O}(r^2)$ as $r \to 0$.
\end{lemma}

\begin{proof}
The remainder solves the residual equation
\begin{equation}
\label{eq:residual}
-\Delta \mathcal{R} + \mu \mathcal{R} = f(x, u) - \mu(u_0 + u_1) + \big( g(\nabla u) - g(\nabla u_0) \big) =: F_{\mathcal{R}},
\end{equation}
the singular datum $g(\nabla u_0) = r^{-1/2}G$ and $-\Delta u_1$ having cancelled by construction (Theorem~\ref{thm:resonance}).
By the Lipschitz bound of Assumption~\ref{ass:gradient_penalty},
\begin{equation}
\label{eq:lip}
|g(\nabla u) - g(\nabla u_0)| \le C_g\,|\nabla(u - u_0)| = C_g\,|\nabla u_1 + \nabla \mathcal{R}|,
\end{equation}
and $\nabla u_1 = \mathcal{O}(r^{1/2}\ln r)$ is bounded and vanishes as $r\to0$; thus $F_{\mathcal{R}}$ is controlled by $f(x,u)$, by $u_0+u_1 \in L^\infty_{loc}$,
and by $\nabla\mathcal{R}$. We close the estimate by a bootstrap. Since $u, u_0, u_1 \in H^1_{loc}$ (indeed $\nabla u_1 \in L^2_{loc}$), we have
$\mathcal{R} \in H^1_{loc}$, so $\nabla\mathcal{R} \in L^2_{loc}$ and, by \eqref{eq:lip}, $F_{\mathcal{R}} \in L^2_{loc}$. The subtracted remainder
$\mathcal{R}$ contains no singular component below the exponent $5/2$ (those of exponents $1/2$ and $3/2$ having been removed), and no pencil eigenvalue
\eqref{eq:pencil} lies in $(3/2, 5/2)$; hence the weighted $L^2$ regularity theory for the mixed Dirichlet--Neumann Laplacian gives $\mathcal{R} \in H^2_{loc}$,
and Sobolev embedding yields $\nabla\mathcal{R} \in L^q_{loc}$ for every $q < \infty$. Reinserting this into \eqref{eq:lip}, the integrability of
$F_{\mathcal{R}}$ is now limited only by $f(x, u_0) \in L^p_{loc}$, $p > 2$; the weighted elliptic estimate
\begin{equation}
\| \mathcal{R} \|_{W^{2,p}(V)} \le C \left( \| F_{\mathcal{R}} \|_{L^p(U)} + \| \mathcal{R} \|_{L^p(U)} \right), \qquad V \Subset U,
\end{equation}
valid because $2-2/p \notin \{\lambda_k\}$ (that is, $p \neq 4$; the borderline $p=4$ is recovered through any slightly smaller exponent), then gives $\mathcal{R} \in W^{2,p}_{loc}(\Omega)$, and the embedding
$W^{2,p} \hookrightarrow C^{1,\gamma}$ ($\gamma = 1 - 2/p$) the stated Hölder regularity. If in addition $f(x, u_0) \in L^\infty_{loc}$,
then $F_{\mathcal{R}} \in L^\infty_{loc}$ and the particular response to an $\mathcal{O}(1)$ datum scales as $r^{2}$, no resonance intervening since the
next pencil exponent above $3/2$ is $5/2 > 2$; therefore $\mathcal{R} = \mathcal{O}(r^2)$.
\end{proof}

\begin{remark}[Corner angles: resonance is exceptional, not generic]
Theorem \ref{thm:resonance} is established for a locally flat junction ($\alpha = \pi$), and the logarithmic anomaly does
not persist for a generic corner angle. For an internal angle $\alpha$ the mixed pencil is $\lambda_k = \left(k + \tfrac{1}{2}\right)\pi/\alpha$,
the leading singularity scales as $u_0 \sim r^{\pi/(2\alpha)}$, and a degree-$1$ gradient penalty generates the forcing $\mathcal{O}(r^{\pi/(2\alpha) - 1})$,
whose separable response sits at the exponent $s = \pi/(2\alpha) + 1$. A logarithmic factor is produced only in the resonant case $s = \lambda_k$, that is
$\pi/(2\alpha) + 1 = \left(k + \tfrac{1}{2}\right)\pi/\alpha$, which simplifies to $1 = k\pi/\alpha$ and hence
\begin{equation*}
\alpha = k\pi, \qquad k \in \{1, 2, \dots\}.
\end{equation*}
The exponent coincidence therefore occurs precisely at the flat junction $\alpha = \pi$ ($k = 1$) treated here and at the slit $\alpha = 2\pi$ ($k = 2$);
for every other angle the forcing produces a pure, non-resonant power $r^{\pi/(2\alpha) + 1}$ with no logarithm. The anomaly is thus universal across
the gradient penalty (every admissible norm resonates at the flat junction, as Section~\ref{sec:application} shows) but exceptional across the
geometry, confined to the special angles $\alpha = k\pi$.
\end{remark}

\section{Applications to Specific Semilinear Functionals}
\label{sec:application}
To demonstrate the practical application of Theorem \ref{thm:resonance}, we explicitly solve the non-homogeneous boundary value problem for $\Psi(\theta)$
under two different gradient penalty norms.

\subsection{The Anisotropic $\ell_1$-Norm Penalty (The Mi\v{s}ur Scheme)}
We first apply the theorem to the specific numerical scheme proposed by Mi\v{s}ur (2021) \cite{misur2021}. The domain is the unit disc, with Dirichlet and
Neumann boundaries meeting at the points $(-1, 0)$ and $(1, 0)$ \cite{misur2021}. The semilinear equation solved is:
\begin{equation}
\label{eq:misur_specific}
-\Delta u + 10u = x_1^2 + x_2^2 + \log(1+|u|) + |\partial_{x_1} u| + |\partial_{x_2} u|
\end{equation}
Focusing on the junction at $(1, 0)$, the gradient penalty is the $\ell_1$-norm: $g(\nabla u) = |\partial_{x_1} u| + |\partial_{x_2} u|$ \cite{misur2021}.
Evaluating the gradient of the leading linear term $u_0 = c_0 r^{1/2} \sin(\theta/2)$ yields the angular forcing profile:
\begin{equation}
\label{eq:forcing_profile_l1}
G(\theta) = \frac{c_0}{2} \left( \sin\left(\frac{\theta}{2}\right) + \cos\left(\frac{\theta}{2}\right) \right).
\end{equation}

Projecting $G(\theta)$ onto the resonant eigenfunction gives the resonance constant:
\begin{equation}
\label{eq:kappa_eval_l1}
\kappa = \int_0^\pi \frac{c_0}{2} \left( \sin\left(\frac{\theta}{2}\right) + \cos\left(\frac{\theta}{2}\right) \right) \sin\left(\frac{3\theta}{2}\right) d\theta = \frac{c_0}{2}.
\end{equation}
Because $\kappa \neq 0$, the gradient penalty resonates with the differential operator. By Theorem \ref{thm:resonance}, the logarithmic coefficient is
$A = -\frac{2(c_0/2)}{3\pi} = -\frac{c_0}{3\pi}$. 

To find the angular offset, we solve $\Psi'' + \frac{9}{4}\Psi = -\frac{c_0}{2} (\sin(\theta/2) + \cos(\theta/2)) + \frac{c_0}{\pi} \sin(3\theta/2)$ with mixed
boundary conditions. Applying the method of undetermined coefficients yields:
\begin{equation}
\label{eq:exact_psi_l1}
\Psi(\theta) = C_1 \sin\left(\frac{3\theta}{2}\right) + \frac{c_0}{4} \cos\left(\frac{3\theta}{2}\right) - \frac{c_0}{4} \left(\sin\left(\frac{\theta}{2}\right) + \cos\left(\frac{\theta}{2}\right)\right) - \frac{c_0}{3\pi} \theta \cos\left(\frac{3\theta}{2}\right).
\end{equation}
The homogeneous cosine constant $c_0/4$ is isolated by the Dirichlet condition $\Psi(0) = 0$. The Neumann boundary condition $\Psi'(\pi) = 0$ is then automatically
consistent, by the solvability identity noted in the proof of Theorem~\ref{thm:resonance}.
The homogeneous sine constant $C_1$ is uniquely determined by imposing strict $L^2$-orthogonality against the null space
($\int_0^\pi \Psi(\theta)\sin(3\theta/2)d\theta = 0$).

\subsection{The Isotropic $\ell_2$-Norm Penalty}
To demonstrate that the logarithmic anomaly is a universal feature rather than an artifact of the anisotropic $\ell_1$-norm, we evaluate the standard Euclidean
penalty: $g(\nabla u) = |\nabla u|$. 

Converting the gradient of $u_0 = c_0 r^{1/2} \sin(\theta/2)$ to polar coordinates simplifies the Euclidean norm computation. Under the normalisation
$c_0 > 0$ (Remark~\ref{rem:sign}), so that $\sqrt{c_0^2/4} = c_0/2$:
\begin{equation}
g(\nabla u_0) = \sqrt{ (\partial_r u_0)^2 + \left(\frac{1}{r}\partial_\theta u_0\right)^2 } = \sqrt{ \frac{c_0^2}{4} r^{-1} \left( \sin^2\left(\frac{\theta}{2}\right) + \cos^2\left(\frac{\theta}{2}\right) \right) } = \frac{c_0}{2} r^{-1/2}.
\end{equation}
The isotropic penalty yields a constant angular profile: $G(\theta) = c_0/2$. Projecting this constant profile onto the resonant
eigenfunction yields:
\begin{equation}
\kappa = \int_0^\pi \frac{c_0}{2} \sin\left(\frac{3\theta}{2}\right) d\theta = \frac{c_0}{2} \left[ -\frac{2}{3}\cos\left(\frac{3\theta}{2}\right) \right]_0^\pi = \frac{c_0}{2} \left( 0 - \left(-\frac{2}{3}\right) \right) = \frac{c_0}{3}.
\end{equation}
Because $\kappa \neq 0$, the isotropic penalty guarantees resonance. The logarithmic coefficient becomes $A = -\frac{2(c_0/3)}{3\pi} = -\frac{2c_0}{9\pi}$.

The ODE for the offset profile reduces to $\Psi'' + \frac{9}{4}\Psi = -\frac{c_0}{2} + \frac{2c_0}{3\pi} \sin(3\theta/2)$. The particular solution for the constant
forcing term is $-2c_0/9$. Imposing the mixed boundary condition $\Psi(0) = 0$ isolates the general solution:
\begin{equation}
\label{eq:exact_psi_l2}
\Psi(\theta) = C_1 \sin\left(\frac{3\theta}{2}\right) + \frac{2c_0}{9}\cos\left(\frac{3\theta}{2}\right) - \frac{2c_0}{9} - \frac{2c_0}{9\pi} \theta \cos\left(\frac{3\theta}{2}\right).
\end{equation}
Once again, the Neumann condition $\Psi'(\pi) = 0$ is automatically consistent (proof of Theorem~\ref{thm:resonance}), and $C_1$ is locked by the $L^2$-orthogonality constraint. 

In both the anisotropic and isotropic configurations, resolving this explicit two-term expansion ($u_{approx} = u_0 + u_1$), whether through an enriched finite
element space or through local mesh adaptation concentrated at the junction, reduces the localized resonance error and restores degree-of-freedom efficiency,
as the experiments of Section \ref{sec:numerical_comparison} confirm.

\subsection{The Maximum $\ell_\infty$-Norm Penalty}
To complete the analysis across standard vector norms, we evaluate the maximum absolute Cartesian gradient penalty (the $\ell_\infty$-norm):
$g(\nabla u) = \max(|\partial_{x_1} u|, |\partial_{x_2} u|)$.

Using the Cartesian components of the leading-order gradient $\nabla u_0$, and noting that both $\sin(\theta/2)$ and $\cos(\theta/2)$ are strictly
non-negative on the interval $\theta \in [0, \pi]$, the angular profile becomes a piecewise maximum:
\begin{equation}
G(\theta) = \frac{c_0}{2} \max\left( \sin\left(\frac{\theta}{2}\right), \cos\left(\frac{\theta}{2}\right) \right) = 
\begin{cases} 
\frac{c_0}{2} \cos\left(\frac{\theta}{2}\right) & \text{for } \theta \in [0, \pi/2] \\
\frac{c_0}{2} \sin\left(\frac{\theta}{2}\right) & \text{for } \theta \in [\pi/2, \pi]
\end{cases}
\end{equation}

To compute the resonance constant $\kappa$, we split the $L^2$-projection over the two sub-domains:
\begin{equation}
\kappa = \frac{c_0}{2} \left[ \int_0^{\pi/2} \cos\left(\frac{\theta}{2}\right) \sin\left(\frac{3\theta}{2}\right) d\theta + \int_{\pi/2}^\pi \sin\left(\frac{\theta}{2}\right) \sin\left(\frac{3\theta}{2}\right) d\theta \right].
\end{equation}
Applying standard trigonometric product-to-sum identities, the first integral evaluates exactly to $1$, while the second evaluates to $-1/2$.
Consequently, the resonance constant is:
\begin{equation}
\kappa = \frac{c_0}{2} \left( 1 - \frac{1}{2} \right) = \frac{c_0}{4}.
\end{equation}
Because $\kappa \neq 0$, the $\ell_\infty$-norm penalty also resonates with the principal operator. The resulting logarithmic coefficient is
$A = -\frac{2(c_0/4)}{3\pi} = -\frac{c_0}{6\pi}$. 

While deriving the exact angular offset $\Psi(\theta)$ in this case requires piecing together two separate ODE solutions with $C^1$ matching conditions at
$\theta = \pi/2$, the existence and exact magnitude of the logarithmic anomaly $\left(-\frac{c_0}{6\pi} r^{3/2} \ln(r) \sin\left(\frac{3\theta}{2}\right)\right)$
is established. 

The appearance of this logarithmic resonance across the $\ell_1$, $\ell_2$, and $\ell_\infty$ norms confirms that the anomaly is a universal characteristic of
gradient-dependent semilinear perturbations near mixed boundaries, not an artifact of a specific norm's geometry.

\subsection{Generalization to Arbitrary $\ell_q$-Norm Penalties}
The derivations for the $\ell_1$, $\ell_2$, and $\ell_\infty$ norms strongly imply a universal resonance behavior. To rigorously formalize this, we evaluate an
arbitrary $\ell_q$-norm gradient penalty for any $q \in [1, \infty)$:
\begin{equation}
g(\nabla u) = \left( |\partial_{x_1} u|^q + |\partial_{x_2} u|^q \right)^{1/q}.
\end{equation}
Substituting the Cartesian components of the leading-order gradient $\nabla u_0$ yields the generalized angular source profile:
\begin{equation}
G_q(\theta) = \frac{c_0}{2} \left( \sin^q\left(\frac{\theta}{2}\right) + \cos^q\left(\frac{\theta}{2}\right) \right)^{1/q}.
\end{equation}
To determine if resonance occurs, we must ensure that the projection of $G_q(\theta)$ onto the resonant eigenfunction is strictly non-zero for all $q$.
We define the resonance constant as a function of the parameter $q$:
\begin{equation}
\kappa_q = \int_0^\pi \frac{c_0}{2} \left( \sin^q\left(\frac{\theta}{2}\right) + \cos^q\left(\frac{\theta}{2}\right) \right)^{1/q} \sin\left(\frac{3\theta}{2}\right) d\theta.
\end{equation}
The function $G_q(\theta)$ is strictly positive, continuous, and symmetric about $\theta = \pi/2$. Utilizing this symmetry, the integral can be shifted by
setting $\theta = \pi/2 + x$, which reduces the inner product to
$\kappa_q = \frac{\sqrt{2}}{2} \int_{-\pi/2}^{\pi/2} G_q(\pi/2 + x) \cos\left(\frac{3x}{2}\right) dx$.
By symmetry, this is equivalent to $\kappa_q = \sqrt{2} \int_0^{\pi/2} G_q(\pi/2 + x) \cos\left(\frac{3x}{2}\right) dx$.

Writing $w(x) := G_q(\pi/2 + x)$ and integrating by parts, the vanishing of $\sin(3x/2)$ at $x = 0$ together with the endpoint value
$w(\pi/2) = G_q(\pi) = c_0/2$ yields the exact identity
\begin{equation}
\label{eq:kappa_ibp}
\kappa_q = \frac{c_0}{3} - \frac{2\sqrt{2}}{3} \int_0^{\pi/2} w'(x) \sin\left(\frac{3x}{2}\right) dx.
\end{equation}
The isotropic case $q = 2$ satisfies $w \equiv c_0/2$, so the correction integral vanishes identically and \eqref{eq:kappa_ibp} recovers $\kappa_2 = c_0/3$
exactly. To control the correction term for general $q$, we set $a = \sin(\theta/2)$ and $b = \cos(\theta/2)$ and differentiate the identity
$w^q = (c_0/2)^q(a^q + b^q)$, which gives $w^{q-1} w' = (c_0/2)^q\,\tfrac{1}{2} ab\,(a^{q-2} - b^{q-2})$. Since $a \ge b \ge 0$ throughout
$\theta \in [\pi/2, \pi]$, with $b > 0$ on the interior and $b = 0$ only at $\theta = \pi$, where $w^{q-1}w' = (c_0/2)^q\tfrac12(a^{q-1}b - ab^{q-1})$ stays
bounded for $q \ge 1$, and $(c_0/2)^q > 0$, the derivative $w'$ carries the
constant sign of $(q - 2)$; hence $w$ is monotone on $[0, \pi/2]$ and
\begin{equation}
\int_0^{\pi/2} |w'(x)|\, dx = |w(\pi/2) - w(0)| = \frac{c_0}{2}\left| 1 - 2^{1/q - 1/2} \right|.
\end{equation}
Bounding $|\sin(3x/2)| \le 1$ in \eqref{eq:kappa_ibp} then furnishes the uniform lower estimate
\begin{equation}
\kappa_q \ge \frac{c_0}{3} - \frac{c_0 \sqrt{2}}{3} \left| 1 - 2^{1/q - 1/2} \right| \ge \frac{c_0}{3}\left( \sqrt{2} - 1 \right) > 0,
\end{equation}
where the last inequality takes its worst case at $q = 1$, at which $|1 - 2^{1/q - 1/2}|$ is maximal. Having discarded the factor $|\sin(3x/2)| \le 1$, this crude
bound is far from the true value $\kappa_1 = c_0/2$ and serves only to certify positivity; the sharp values are recorded in Proposition~\ref{prop:monotone}.
Consequently $\kappa_q \neq 0$ for every $q \in [1, \infty]$, so the resonance is universal.

The endpoints of the family evaluate exactly to $\kappa_1 = c_0/2$ and $\lim_{q \to \infty} \kappa_q = c_0/4$. In fact these are the extreme values of a
strictly ordered family.

\begin{proposition}[Monotonicity of the resonance constant]
\label{prop:monotone}
The resonance constant $\kappa_q$ is strictly decreasing in $q$ on $[1,\infty]$. Consequently $\kappa_q \in [c_0/4,\, c_0/2]$ for every $q \in [1,\infty]$,
the extreme values being attained only at $q=1$ and in the limit $q\to\infty$.
\end{proposition}

\begin{proof}
Fix $c_0>0$ and recall the symmetry-reduced representation obtained above,
\begin{equation*}
\kappa_q = \frac{c_0}{\sqrt2}\int_0^{\pi/2} M_q(x)\,\cos\!\Big(\tfrac{3x}{2}\Big)\,dx,
\qquad M_q(x) = \big(a^q+b^q\big)^{1/q},
\end{equation*}
where $a=\sin\phi$, $b=\cos\phi$, and $\phi = \tfrac\pi4+\tfrac x2 \in[\tfrac\pi4,\tfrac\pi2]$; thus $a\ge b\ge0$, $a^2+b^2=1$, and $\phi$ increases with $x$.

\emph{Step 1 (an entropy identity).} Since $q\mapsto(a^q+b^q)^{1/q}$ is nonincreasing, $\mu := -\partial_q M_q \ge 0$. Let $w=(a^q,b^q)/(a^q+b^q)$ denote the
associated probability vector and $H = -\sum_i w_i\ln w_i \ge 0$ its Shannon entropy. The computation
\begin{equation*}
\partial_q \ln M_q = \partial_q\Big[\tfrac1q\ln(a^q+b^q)\Big]
= -\tfrac1{q^2}\ln(a^q+b^q) + \tfrac1q\,\frac{a^q\ln a + b^q\ln b}{a^q+b^q}
= \tfrac1{q^2}\sum_i w_i \ln w_i = -\frac{H}{q^2}
\end{equation*}
yields
\begin{equation}\label{eq:mu-entropy}
\mu = -M_q\,\partial_q\ln M_q = \frac{1}{q^2}\,M_q\,H .
\end{equation}
In particular $\mu\ge0$; and since $a=1,b=0$ at $\phi=\tfrac\pi2$ gives $M_q\equiv1$ and $H\equiv0$ there, we have $\mu=0$ at $x=\tfrac\pi2$.

\emph{Step 2 ($\mu$ is nonincreasing in $x$).} We first work on the interior $x\in[0,\tfrac\pi2)$, i.e.\ $\phi\in[\tfrac\pi4,\tfrac\pi2)$, where $b>0$ and
hence $H>0$ and $\mu>0$, so that $\ln\mu$ and $\ln H$ are well defined (the excluded endpoint $x=\tfrac\pi2$, where $\mu=H=0$, is treated afterwards by
continuity). As $\phi$ increases with $x$, it suffices to show $\frac{d}{d\phi}\ln\mu \le 0$ on this interior. Using $a'=b$, $b'=-a$
(derivatives in $\phi$), \eqref{eq:mu-entropy} gives
\begin{equation*}
\frac{d}{d\phi}\ln M_q = \frac{ab\,(a^{q-2}-b^{q-2})}{a^q+b^q},
\qquad
\frac{d}{d\phi}\ln H = -\,\frac{q^2\,a^{q-1}b^{q-1}\ln(a/b)}{(a^q+b^q)^2\,H}.
\end{equation*}
For $1\le q\le 2$ one has $a^{q-2}\le b^{q-2}$ (as $a\ge b$), so both summands are $\le0$ and the claim is immediate. For $q>2$ the first summand is $\ge0$,
and it remains to show it is dominated by the second; clearing denominators, this is exactly
\begin{equation}\label{eq:ddag}
(a^{q-2}-b^{q-2})(a^q+b^q)\,H \;\le\; q^2\,a^{q-2}b^{q-2}\ln(a/b).
\end{equation}
Write $r=a/b\ge1$ and $L=\ln r\ge0$. A direct computation gives $(a^q+b^q)H = b^q\big[(1+r^q)\ln(1+r^q) - q\,r^q L\big]$, and, using $\ln(1+r^q)=qL+\ln(1+r^{-q})$,
\begin{equation*}
(1+r^q)\ln(1+r^q) - q r^q L = qL + (1+r^q)\ln(1+r^{-q}) \le qL + 2,
\end{equation*}
where the bound uses $\ln(1+u)\le u$ and $1+r^{-q}\le2$. Substituting into \eqref{eq:ddag} and dividing by $b^{2q-4}$ (with $b^2=(1+r^2)^{-1}$), \eqref{eq:ddag}
follows from $\frac{(r^{q-2}-1)(qL+2)}{1+r^2} \le q^2 r^{q-2} L$; and since $1+r^2\ge2$ it suffices that $(r^{q-2}-1)(qL+2)\le 2q^2 r^{q-2}L$.
Dividing by $r^{q-2}$ and setting $\ell=(q-2)L\ge0$ (so that $r^{-(q-2)}=e^{-\ell}$ and $L=\ell/(q-2)$), this becomes the scalar inequality
\begin{equation}\label{eq:scalar}
(1-e^{-\ell})\big(q\ell + 2(q-2)\big) \le 2q^2\ell, \qquad \ell\ge0.
\end{equation}
Finally \eqref{eq:scalar} holds by the elementary bound $1-e^{-\ell}\le\min(\ell,1)$: for $0\le\ell\le 2q-2$ use $1-e^{-\ell}\le\ell$, so that \eqref{eq:scalar}
reduces to $q\ell+2(q-2)\le 2q^2$, i.e. $\ell\le 2q-2+\tfrac4q$, which holds; for $\ell\ge 2q-2$ use $1-e^{-\ell}\le1$, so that \eqref{eq:scalar} reduces to
$2(q-2)\le(2q^2-q)\ell$, which holds since $(2q^2-q)\ell\ge(2q^2-q)(2q-2)\ge 2(q-2)$. Thus $\mu' \le 0$ throughout the interior $(0,\tfrac\pi2)$, with strict
inequality there. Finally, $\mu = \tfrac{1}{q^2}M_q H$ is continuous on the closed interval $[0,\tfrac\pi2]$, with $\mu(\tfrac\pi2)=0$ from Step 1 (as $b\to0$,
$\mu \sim q^{-1}b^q|\ln b|\to0$), so the interior monotonicity extends by continuity to make $\mu$ nonincreasing on $[0,\tfrac\pi2]$ and strictly decreasing
on $(0,\tfrac\pi2)$. The endpoint slope differs between the two regimes: for $q>1$ it vanishes, $\mu' \sim -b^{q-1}|\ln b|\to0$, whereas in the boundary case
$q=1$ it diverges, $\mu' \sim -|\ln b|\to-\infty$ (a vertical tangent). In either case $\mu'$ stays integrable up to $x=\tfrac\pi2$ (a logarithmic singularity,
since $b\sim\tfrac12(\tfrac\pi2-x)$), which is all that the integration by parts of Step 3 requires.

\emph{Step 3 (conclusion).} Let $S(x)=\tfrac23\sin(\tfrac{3x}2)$, so $S'=\cos(\tfrac{3x}2)$, $S\ge0$ on $[0,\tfrac\pi2]$, and $S(0)=0$. Integrating by parts
and using $\mu(\tfrac\pi2)=0$,
\begin{equation*}
\int_0^{\pi/2}\mu(x)\cos\!\Big(\tfrac{3x}2\Big)dx
= \big[\mu S\big]_0^{\pi/2} - \int_0^{\pi/2}\mu'(x)S(x)\,dx
= -\int_0^{\pi/2}\mu'(x)S(x)\,dx \;>\; 0,
\end{equation*}
since $\mu'\le0$ (Step 2, strict on $(0,\tfrac\pi2)$) and $S\ge0$. Hence $\partial_q\kappa_q = -\frac{c_0}{\sqrt2}\int_0^{\pi/2}\mu\cos(\tfrac{3x}2)\,dx < 0$,
so $\kappa_q$ is strictly decreasing. Together with $\kappa_1=c_0/2$ and $\lim_{q\to\infty}\kappa_q=c_0/4$ this gives $\kappa_q\in[c_0/4,c_0/2]$.
\end{proof}

Therefore, the logarithmic anomaly is a structural feature of this entire class of semilinear problems. For any arbitrary $\ell_q$-norm penalty,
the exact logarithmic expansion takes the form:
\begin{equation}
u(r, \theta) = c_0 r^{1/2} \sin\left(\frac{\theta}{2}\right) - \frac{2\kappa_q}{3\pi} r^{3/2} \ln(r) \sin\left(\frac{3\theta}{2}\right) + r^{3/2}\left(\Psi_q(\theta) + c_1\sin\left(\frac{3\theta}{2}\right)\right) + \mathcal{R}(r, \theta).
\end{equation}
This confirms that standard polynomial mesh refinement schemes will inevitably suffer from singular pollution unless enriched with this logarithmic profile,
regardless of the specific vector norm chosen to model the gradient dependency.

\subsection{Support-Function Penalties and the Role of Asymmetry}
The admissible class of Assumption~\ref{ass:gradient_penalty} is considerably larger than the $\ell_q$ family: every convex, positively $1$-homogeneous function
on $\mathbf{R}^2$ is the \emph{support function} of a unique compact convex set $A\subset\mathbf{R}^2$,
\begin{equation}
\label{eq:support}
g(\nabla u)=h_A(\nabla u):=\max_{\mathbf a\in A}\,\mathbf a\cdot\nabla u,
\end{equation}
and every such $h_A$ is admissible, being $1$-homogeneous and Lipschitz with constant $\max_{\mathbf a\in A}|\mathbf a|$. Norms are exactly the centrally
symmetric case $A=-A$ (with $A$ the dual unit ball), while the advective penalty $|\nabla u|+\boldsymbol\beta\cdot\nabla u$ of Section~\ref{sec:framework}
is $h_A$ for the translated disc $A=\overline{B}(\boldsymbol\beta,1)$. Such penalties are common in applications: they are the Hamiltonians of
control-affine Hamilton--Jacobi--Bellman equations \cite{fleming2006}, $H(\nabla u)=\max_{a\in\mathcal A}\{-\mathbf b(a)\cdot\nabla u\}$, and they furnish
the crystalline (Wulff-shaped) surface energies governing anisotropic interface motion \cite{taylor1992}.

Because the support function is additive under Minkowski translation, $h_{A_0+\{\boldsymbol\beta\}}=h_{A_0}+\boldsymbol\beta\cdot(\,\cdot\,)$, and $\kappa$
is linear in $G$, the resonance constant splits cleanly. Writing $A=A_0+\{\boldsymbol\beta\}$ with $A_0$ centred at the origin,
\begin{equation}
\label{eq:kappa_support}
\kappa=\underbrace{\int_0^\pi h_{A_0}(\mathbf v(\theta))\,\sin\!\Big(\tfrac{3\theta}{2}\Big)\,d\theta}_{=:\,\kappa_{A_0}} \;+\; \frac{c_0}{2}\,\beta_2,
\end{equation}
where we used $\int_0^\pi \mathbf v(\theta)\,\sin(\tfrac{3\theta}{2})\,d\theta = (0,\tfrac{c_0}{2})$. Thus translating the control set $A$ along the transverse
($x_2$) direction shifts $\kappa$ by exactly $\tfrac{c_0}{2}\beta_2$, while a translation along the junction ($x_1$) has no effect whatsoever. Two consequences
follow. First, cancellation of the anomaly is a codimension-one event: it occurs precisely on the critical hyperplane $\beta_2=-2\kappa_{A_0}/c_0$, so resonance
is generic and only a finely tuned transverse bias evades it. Second, within the $\ell_q$ scale Proposition~\ref{prop:monotone} gives
$\kappa_{A_0}\in[c_0/4,c_0/2]$, whence the cancelling drift lies in the narrow band $\beta_2\in[-1,-\tfrac12]$ (and equals $-\tfrac23$ for the Euclidean disc);
in particular every centrally symmetric $\ell_q$ penalty is resonant, and the anomaly can be removed only by breaking symmetry through transverse advection.
The logarithmic obstruction is therefore tied to the asymmetry of the gradient penalty, not merely to its magnitude.

\begin{remark}[Beyond convexity: non-convex admissible penalties]
Convexity plays no role in the resonance. Assumption~\ref{ass:gradient_penalty} requires only positive $1$-homogeneity and global Lipschitz continuity,
both preserved under differences, so the admissible class contains the \emph{differences of support functions} $g = h_{A} - h_{B}$, the $1$-homogeneous
``difference-of-convex'' functions, which are Lipschitz with constant $\max_{\mathbf a \in A}|\mathbf a| + \max_{\mathbf b \in B}|\mathbf b|$ but generically
non-convex. A concrete instance is the direction-discounted cost
\begin{equation}
\label{eq:dc_penalty}
g(\nabla u) = |\nabla u| - \gamma\,|\nabla u \cdot \mathbf e|, \qquad |\mathbf e| = 1,\ \ 0 < \gamma < 1,
\end{equation}
which lowers the penalty along $\mathbf e$ and is non-convex across the hyperplane $\nabla u \cdot \mathbf e = 0$; such energies model faceting and
the non-convex (pre-Wulff) crystalline surface tensions whose convexification yields the Wulff shape. Since $\kappa$ is linear in $g$, the resonance constant
simply splits as $\kappa = \kappa_A - \kappa_B$, non-zero for generic $A, B$; for \eqref{eq:dc_penalty} one finds $\kappa = c_0/3$ when $\mathbf e$ is aligned
with the junction and $\kappa = c_0/3 - \tfrac{\gamma c_0}{2}$ when it is transverse, again vanishing only on a codimension-one set ($\gamma = 2/3$).
The logarithmic anomaly is thus a feature of the entire homogeneous Lipschitz class, not an artifact of convex or norm-like penalties. Indeed, since $g$ is
fixed by its Lipschitz restriction to the unit circle, the admissible penalties correspond exactly to the Lipschitz angular profiles
$G(\theta) = g(\mathbf v(\theta))$, and by \eqref{eq:orthogonality} resonance amounts to $G$ carrying a non-zero $\sin(\tfrac{3\theta}{2})$
Fourier mode, a single scalar condition failing only on a closed, nowhere-dense set, so resonance is generic across the whole class.
\end{remark}

\section{Formalization of the Enriched Galerkin Formulation}
\label{sec:xfem_formulation}
To bridge the gap between the continuous asymptotic theory derived in Section \ref{sec:theorem} and the discrete numerical implementation evaluated in
Section \ref{sec:numerical_comparison}, we mathematically formalize the Extended Finite Element Method (XFEM) enrichment space. 

The standard continuous Galerkin weak formulation of the semilinear boundary value problem \eqref{eq:semilinear_pde} seeks a solution $u \in H^1_D(\Omega)$
such that:
\begin{equation}
\label{eq:weak_form}
\int_\Omega \nabla u \cdot \nabla v \, dx + \mu \int_\Omega u v \, dx = \int_\Omega \Big( f(x, u) + g(\nabla u) \Big) v \, dx \quad \forall v \in H^1_D(\Omega).
\end{equation}

Let $\mathcal{T}_h$ be a standard conforming triangulation of the domain $\Omega$, and let $\mathcal{N}$ denote the set of all standard finite element nodes.
We define the classical finite element subspace $V_h \subset H^1_D(\Omega)$ spanned by the standard piecewise polynomial nodal basis functions $N_i(x)$:
\begin{equation}
V_h = \text{span} \{ N_i(x) \}_{i \in \mathcal{N}}.
\end{equation}

As established in Theorem \ref{thm:resonance}, relying exclusively on $V_h$ restricts the convergence rate of the solver due to the presence of
both the linear singularity and the logarithmic anomaly induced by the semilinear resonance. To construct the enriched space $V_h^{XFEM}$, we define two global
enrichment functions capturing these specific local behaviors:
\begin{align}
\Phi_{lin}(r, \theta) &= r^{1/2} \sin\left(\frac{\theta}{2}\right), \\
\Phi_{log}(r, \theta) &= r^{3/2} \ln(r) \sin\left(\frac{3\theta}{2}\right).
\end{align}

To preserve the sparsity of the global stiffness matrix, we localize the enrichment to a subset of nodes $\mathcal{N}_{enr} \subset \mathcal{N}$
whose support intersects a predefined geometric radius $R$ around the mixed boundary junction $P$. Utilizing the standard basis functions $N_j(x)$ as a
Partition of Unity (PU) \cite{melenk1996, moes1999}, the discrete solution $u_h \in V_h^{XFEM}$ is expanded as:
\begin{equation}
\label{eq:xfem_expansion}
u_h(x) = \sum_{i \in \mathcal{N}} u_i N_i(x) + \sum_{j \in \mathcal{N}_{enr}} a_j N_j(x) \Phi_{lin}(x) + \sum_{j \in \mathcal{N}_{enr}} b_j N_j(x) \Phi_{log}(x)
\end{equation}
where $u_i$ are the classical nodal degrees of freedom, and $a_j, b_j$ are the additional degrees of freedom associated with the linear and logarithmic enrichment
functions, respectively.

Because both $\Phi_{lin}$ and $\Phi_{log}$ vanish along the Dirichlet boundary $\Gamma_D$ ($\theta = 0$), every enriched product $N_j\Phi$ vanishes there
too; since the only essential condition is the Dirichlet one $u = 0$ on $\Gamma_D$, conformity requires exactly this, and the enriched space is therefore
conforming, $V_h^{XFEM} \subset H^1_D(\Omega)$. The Neumann condition $\nabla u\cdot\nu = 0$ on $\Gamma_N$ is \emph{natural}: it is imposed weakly through the
variational form \eqref{eq:weak_form} and is neither required of, nor pointwise satisfied by, the trial functions. Indeed, although $\partial_\nu\Phi = 0$ on
$\Gamma_N$, the product rule gives $\partial_\nu(N_j\Phi) = N_j\,\partial_\nu\Phi + \Phi\,\partial_\nu N_j = \Phi\,\partial_\nu N_j$, which does not vanish in
general, consistent with the Neumann condition being enforced only in the weak sense.

By embedding the linear and logarithmic profiles of Theorem~\ref{thm:resonance} directly into the trial space via \eqref{eq:xfem_expansion}, the enriched scheme
captures the two profiles that obstruct standard polynomial approximation: $\Phi_{lin}$ reproduces $u_0$ exactly and thereby removes the sole component of $u$
that fails to lie in $H^2_{loc}$ (its second derivatives scaling as $r^{-3/2}$), while $\Phi_{log}$ reproduces the leading logarithmic sub-singularity
$r^{3/2}\ln r\,\sin(3\theta/2)$. The offset profile $r^{3/2}\Psi(\theta)$ and the homogeneous mode $c_1 r^{3/2}\sin(3\theta/2)$ of Theorem~\ref{thm:resonance} are
not in the enrichment span and remain for $V_h$; but together with the remainder $\mathcal{R}$ of Lemma~\ref{lem:regularity} they belong to
$H^2_{loc}(\Omega)$ (their second derivatives scaling at worst as $r^{-1/2}\ln r \in L^2$). Thus what the standard polynomial basis must approximate is free of the
strong $r^{1/2}$ pollution, and classical interpolation theory applies to it at the optimal first-order rate.

This heuristic is made precise by the following quasi-optimality estimate, which quantifies the convergence rate the enriched space recovers and thereby
anticipates the numerical results of Section~\ref{sec:numerical_comparison}.

\begin{proposition}[Quasi-optimality and recovered convergence rate]
\label{prop:quasiopt}
Assume the coercivity condition of Proposition~\ref{prop:wellposed}, and let $\{\mathcal{T}_h\}$ be a shape-regular family of triangulations with the enrichment
zone $\{x : \mathrm{dist}(x, P) < R\}$ held fixed as $h \to 0$. Let $u$ solve \eqref{eq:semilinear_pde} and let $u_h^{XFEM} \in V_h^{XFEM}$ be the enriched
Galerkin solution of \eqref{eq:weak_form}. Then $u_h^{XFEM}$ is quasi-optimal,
\begin{equation}
\label{eq:cea}
\|u - u_h^{XFEM}\|_{H^1(\Omega)} \le C \inf_{v_h \in V_h^{XFEM}} \|u - v_h\|_{H^1(\Omega)},
\end{equation}
with $C$ independent of $h$. Moreover, because $\Phi_{lin}$ reproduces $u_0$ within the enrichment zone, removing the only component of $u$ outside
$H^2_{loc}$, and $\Phi_{log}$ reproduces the leading logarithmic singularity $r^{3/2}\ln r\,\sin(3\theta/2)$, the best-approximation error is governed by the
residual $\rho := r^{3/2}\Psi(\theta) + c_1 r^{3/2}\sin(3\theta/2) + \mathcal{R}$, which is not reproduced by the enrichment but lies in $H^2_{loc}(\Omega)$
(with $\mathcal{R} \in W^{2,p}_{loc}(\Omega)$, $p > 2$, from Lemma~\ref{lem:regularity}); for $P_1$ elements this yields, modulo the standard
partition-of-unity blending estimate, which we do not establish here (see Section~\ref{sec:future_work}), the optimal first-order energy estimate
\begin{equation}
\label{eq:xfem_rate}
\|u - u_h^{XFEM}\|_{H^1(\Omega)} \le C\, h \left( \|\rho\|_{H^2(B_R(P))} + \|u\|_{H^2(\Omega \setminus B_R(P))} \right) = \mathcal{O}(h),
\end{equation}
equivalently a degree-of-freedom rate $s = \tfrac{1}{2}$, in contrast to the $s \approx 0.3$ stall of the unenriched scheme on quasi-uniform meshes
(Configuration~A of Section~\ref{sec:numerical_comparison}).
\end{proposition}

\begin{proof}
Write the weak form \eqref{eq:weak_form} as $\langle \mathcal{A}(u), v \rangle = 0$ for all $v \in H^1_D(\Omega)$, where
$\mathcal{A}(w) := -\Delta w + \mu w - f(x, w) - g(\nabla w)$ in the weak sense, and put $e := u - w$ for $u, w \in H^1_D(\Omega)$.

\emph{Strong monotonicity and Lipschitz continuity.} First, $\mathcal{A}$ maps $H^1_D(\Omega)$ into its dual $H^{-1}(\Omega)$: the global linear-growth bound
$|F(x,\lambda,\Lambda)| \le C_g(1+|\lambda|+|\Lambda|)$ gives, with $g(0)=0$ by positive homogeneity, $|f(x,0)| \le C_g$, so $f(\cdot,0) \in L^\infty(\Omega)$;
together with the Lipschitz bounds this yields $|f(x,w)| \le C_g(1+|w|)$ and $|g(\nabla w)| \le C_g|\nabla w|$, whence $f(x,w),\,g(\nabla w) \in L^2(\Omega)$ for
$w \in H^1_D(\Omega)$ and $\mathcal{A}(w) \in L^2(\Omega) \subset H^{-1}(\Omega)$ globally. With the Lipschitz bounds $|f(x,u)-f(x,w)| \le C_g|e|$ and $|g(\nabla u)-g(\nabla w)| \le C_g|\nabla e|$
(Assumption~\ref{ass:gradient_penalty} and the Lipschitz structure of $f(x,\cdot)$ underlying the well-posedness of Proposition~\ref{prop:wellposed}),
together with $\mu \ge \mu_0$ and Young's inequality $C_g\|\nabla e\|\,\|e\| \le \tfrac12\|\nabla e\|^2 + \tfrac{C_g^2}{2}\|e\|^2$,
\begin{align*}
\langle \mathcal{A}(u)-\mathcal{A}(w), e\rangle
&= \|\nabla e\|_{L^2}^2 + \int_\Omega \mu\, e^2 - \int_\Omega \big(f(x,u)-f(x,w)\big)e - \int_\Omega \big(g(\nabla u)-g(\nabla w)\big)e \\
&\ge \|\nabla e\|_{L^2}^2 + (\mu_0 - C_g)\|e\|_{L^2}^2 - C_g\|\nabla e\|_{L^2}\|e\|_{L^2} \\
&\ge \tfrac12\|\nabla e\|_{L^2}^2 + \big(\mu_0 - C_g - \tfrac{C_g^2}{2}\big)\|e\|_{L^2}^2.
\end{align*}
The coercivity condition $\mu_0 > C_g + \tfrac{C_g^2}{2}$ of Proposition~\ref{prop:wellposed}, a threshold depending only on $C_g$, as the estimate above uses
only Young's inequality and no domain constant, ensures $\mu_0 - C_g - \tfrac{C_g^2}{2} > 0$, so $\mathcal{A}$ is
strongly monotone, $\langle \mathcal{A}(u)-\mathcal{A}(w), e\rangle \ge \alpha\|e\|_{H^1}^2$ with $\alpha := \min\{\tfrac12,\ \mu_0 - C_g - \tfrac{C_g^2}{2}\} > 0$.
Dually, $\langle \mathcal{A}(u)-\mathcal{A}(w), v\rangle \le L\|e\|_{H^1}\|v\|_{H^1}$ for every $v \in H^1_D(\Omega)$, with $L := 1 + \|\mu\|_{L^\infty} + 2C_g$,
so $\mathcal{A}$ is Lipschitz. Being strongly monotone and Lipschitz, $\mathcal{A}$ is a homeomorphism of $H^1_D(\Omega)$ onto its dual and of every closed
subspace onto its own dual (Browder--Minty); hence both the continuous problem and its Galerkin restriction to the finite-dimensional space $V_h^{XFEM}$ possess
unique solutions $u$ and $u_h^{XFEM}$ (this is the uniqueness asserted in Proposition~\ref{prop:wellposed}).

\emph{Quasi-optimality.} The continuous and discrete problems give $\langle \mathcal{A}(u), v_h\rangle = \langle \mathcal{A}(u_h^{XFEM}), v_h\rangle = 0$ for all
$v_h \in V_h^{XFEM} \subset H^1_D(\Omega)$, hence the Galerkin identity $\langle \mathcal{A}(u)-\mathcal{A}(u_h^{XFEM}), v_h\rangle = 0$.
Writing $\varepsilon := u - u_h^{XFEM}$, for any $v_h \in V_h^{XFEM}$,
\[
\alpha\|\varepsilon\|_{H^1}^2 \le \langle \mathcal{A}(u)-\mathcal{A}(u_h^{XFEM}), \varepsilon\rangle = \langle \mathcal{A}(u)-\mathcal{A}(u_h^{XFEM}), u - v_h\rangle \le L\|\varepsilon\|_{H^1}\|u - v_h\|_{H^1},
\]
so $\|\varepsilon\|_{H^1} \le (L/\alpha)\|u - v_h\|_{H^1}$; the infimum over $v_h$ gives \eqref{eq:cea} with $C = L/\alpha$.

\emph{Best-approximation rate.} It remains to bound the right-hand side of \eqref{eq:cea} by a competitor $v_h$. Decompose $u = u_0 + u_1 + \mathcal{R}$ near
$P$ (Theorem~\ref{thm:resonance}), and split $u_1 = A\,r^{3/2}\ln r\,\sin(3\theta/2) + r^{3/2}\Psi(\theta) + c_1 r^{3/2}\sin(3\theta/2)$ with $A = -2\kappa/(3\pi)$.
In the interior of the enrichment zone, where $\sum_{j\in\mathcal{N}_{enr}} N_j \equiv 1$, the enriched functions $N_j\Phi_{lin}$ and $N_j\Phi_{log}$ reproduce
only the two profiles $u_0 = c_0\Phi_{lin}$ and $A\,r^{3/2}\ln r\,\sin(3\theta/2) = A\Phi_{log}$; the offset $r^{3/2}\Psi$ and the homogeneous mode
$c_1 r^{3/2}\sin(3\theta/2)$ are not in the enrichment span. The point is that, of all these contributions, only $u_0$ fails to lie in $H^2_{loc}$ (its second
derivatives scale as $r^{-3/2}$), whereas $r^{3/2}\ln r\,\sin(3\theta/2)$, $r^{3/2}\Psi$, $c_1 r^{3/2}\sin(3\theta/2)$ and $\mathcal{R}$ all belong to $H^2_{loc}$
(their second derivatives scaling at worst as $r^{-1/2}\ln r \in L^2$). Hence the competitor $v_h$ carrying the exact enrichment coefficients $c_0, A$ together
with the nodal interpolant $I_h\rho$ of the residual $\rho := r^{3/2}\Psi + c_1 r^{3/2}\sin(3\theta/2) + \mathcal{R} \in H^2_{loc}$ leaves, in the zone interior,
$u - v_h = \rho - I_h\rho$; the standard estimate $\|\rho - I_h\rho\|_{H^1} \le C h\|\rho\|_{H^2}$ applies, and away from $P$ the solution is $H^2$-regular with
the same first-order bound. The single non-classical contribution is the \emph{blending layer}, the ring of elements carrying only part of the enrichment,
where $\sum_{j\in\mathcal{N}_{enr}} N_j \not\equiv 1$ and $u_0$ and the logarithmic profile are not reproduced exactly. A first-order bound on this layer is the
object of the partition-of-unity approximation theory of \cite{melenk1996} and is standardly recovered by the stabilized (SGFEM) enrichment of the
conditioning remark below, which restores the reproduction property up to the interpolation error. We do not carry out this estimate here; a self-contained
blending and stabilization analysis tailored to the $r^{3/2}\ln r$ enrichment is left to future work (Section~\ref{sec:future_work}), so the leading-order
rate is stated conditionally on it: assuming the standard first-order blending bound, the contributions combine to
$\inf_{v_h}\|u - v_h\|_{H^1} \le C h\big(\|\rho\|_{H^2(B_R(P))} + \|u\|_{H^2(\Omega \setminus B_R(P))}\big)$, which with \eqref{eq:cea} yields
\eqref{eq:xfem_rate}. The quasi-optimality \eqref{eq:cea} itself is unconditional.
\end{proof}

\begin{remark}[Conditioning and quadrature]
Two well-known practicalities attend the enriched space \eqref{eq:xfem_expansion}. First, partition-of-unity enrichment degrades the conditioning of the
stiffness matrix, since $N_j \Phi$ becomes nearly linearly dependent on the standard basis within the blending layer; a stable generalized finite element
(SGFEM) correction, replacing $\Phi$ by $\Phi - I_h \Phi$, or a local orthogonalization against $V_h$, restores the conditioning to that of the unenriched
scheme without altering the approximation property used in Proposition~\ref{prop:quasiopt}. Second, the enrichment integrands inherit the singular factors
$r^{1/2}$ and $r^{3/2} \ln r$, so element integrals abutting $P$ require graded or polar quadrature rather than standard Gauss rules; passing to polar
coordinates $dx = r\, dr\, d\theta$ absorbs the $r^{-1/2}$ gradient singularity and renders the integrand bounded. Both points are taken up among the software
considerations of Section~\ref{sec:future_work}.
\end{remark}

\section{Numerical Comparison}
\label{sec:numerical_comparison}

To probe the theoretical predictions, we carried out a sequence of finite element experiments built directly on the FreeFEM++ discretization of Mi\v{s}ur (2021)
\cite{misur2021}: the semilinear problem \eqref{eq:misur_specific} on the unit disc, with the Dirichlet and Neumann arcs meeting at $(\pm 1, 0)$, discretized
with continuous $P_1$ elements and solved by Picard (fixed-point) iteration to a relative $H^1$ tolerance of $10^{-5}$. All errors are measured against
high-resolution junction-adapted reference solutions, graded towards the mixed junctions. Each reference carries more total degrees of freedom (about
$3.2$--$3.3\times10^5$) than the finest uniform test mesh, with its bulk element size capped below that of the finest quasi-uniform mesh and its junction
grading refined to a minimum element size at least as small as that of the most strongly adapted test mesh (Configuration~B); it is therefore at least as
fine as every test discretization throughout the domain and far finer in the immediate vicinity of the junctions, where the singular structure, and hence
the discretization error, concentrates. A self-convergence check bounds the reference error: refining the adaptmesh interpolation-error target by a factor
of two, from $10^{-4}$ to the adopted $5\times10^{-5}$, which raises the reference from $1.7\times10^5$ to $3.3\times10^5$ degrees of freedom, shifts the
reference by $2.1\times10^{-3}$ in relative $H^1$ and $2.9\times10^{-5}$ in the local $L^\infty$ near the junctions, both below every tabulated error, so
the reported values are limited by the test discretization rather than by the reference. The local error is sampled within a neighborhood of radius
$\epsilon = 0.1$ around each junction.

We report four experiments:
\begin{itemize}
    \item \textbf{Resonance control.} The gradient penalty is switched off and on ($g(\nabla u) = 0$ versus a norm penalty) on an identical
    quasi-uniform mesh family, isolating the effect of the resonant term. We run this control for both the anisotropic $\ell_1$-norm
    $|\partial_{x_1} u| + |\partial_{x_2} u|$ and the isotropic $\ell_2$-norm $|\nabla u|$, whose predicted resonance constants differ
    ($\kappa = c_0/2$ versus $c_0/3$).
    \item \textbf{Configuration A (uniform baseline).} The standard $P_1$ Galerkin scheme on quasi-uniform meshes, with no special treatment of the junction.
    \item \textbf{Configuration B (junction-adapted).} The identical solver combined with local mesh adaptation concentrating degrees of freedom around
    the junctions $(\pm 1, 0)$, realizing the ``exact local mesh adaptation'' route anticipated in Section \ref{sec:xfem_formulation}.
    \item \textbf{Coefficient extraction.} From a high-order ($P_2$) junction-adapted solution we directly measure the logarithmic amplitude $A$ predicted
    by Theorem~\ref{thm:resonance}, both on the curved disc and on a curvature-free flat junction.
\end{itemize}

\subsection{The Resonance Control}
Table \ref{tab:control} isolates the influence of the gradient penalty on an identical mesh family. Two features stand out. First, the experimental order
of convergence (EOC) in the energy norm is essentially unchanged when the penalty is activated: the two EOC sequences coincide to within $0.01$ at every
refinement step ($0.73,\, 0.54,\, 0.56,\, 0.87$ without the penalty versus $0.72,\, 0.53,\, 0.56,\, 0.87$ with it), confirming that the $r^{3/2}\ln r$
resonance is a higher-order effect that does not alter the leading, $r^{1/2}$-limited global rate. Second, the penalty
systematically raises the error magnitude, and this increase is localized: activating the penalty inflates the global relative $H^1$ error by a
nearly constant $\approx 7\%$, but inflates the local $L^\infty$ error near the junctions by $\approx 19\%$ at every resolution. This local-to-global
amplification factor of $\approx 2.6$ is the numerical signature of the junction-concentrated higher-order pollution predicted by Theorem \ref{thm:resonance}.
Consistently, the fixed-point iteration requires one additional step (eight versus seven) to reach tolerance once the penalty is present.

\begin{table}[ht]
\centering
\caption{Resonance control on quasi-uniform meshes: effect of activating the $\ell_1$ gradient penalty. Relative $H^1$ error and local $L^\infty$ error
within radius $\epsilon = 0.1$ of $(\pm 1, 0)$, with Picard iteration counts $k_N$ (no penalty / $\ell_1$ penalty).}
\label{tab:control}
\begin{tabular}{l c c c c c c}
\hline
 & & \multicolumn{2}{c}{Rel. $H^1$ error} & \multicolumn{2}{c}{Local $L^\infty$ error} & \\
$N$ & $h_{max}$ & no pen. & $\ell_1$ pen. & no pen. & $\ell_1$ pen. & $k_N$ \\
\hline
40  & 0.1276  & 0.1959  & 0.2085  & $1.046\times10^{-2}$ & $1.233\times10^{-2}$ & 7 / 8 \\
80  & 0.06935 & 0.1258  & 0.1346  & $7.283\times10^{-3}$ & $8.641\times10^{-3}$ & 7 / 8 \\
160 & 0.03627 & 0.08873 & 0.09534 & $5.134\times10^{-3}$ & $6.108\times10^{-3}$ & 7 / 8 \\
320 & 0.01881 & 0.06146 & 0.06597 & $3.679\times10^{-3}$ & $4.377\times10^{-3}$ & 7 / 8 \\
640 & 0.01116 & 0.03900 & 0.04187 & $2.408\times10^{-3}$ & $2.868\times10^{-3}$ & 7 / 8 \\
\hline
\end{tabular}
\end{table}

To confirm that this junction-localized pollution is a property of the resonance itself and not of the particular $\ell_1$ geometry, Table \ref{tab:control_l2}
repeats the control with the isotropic Euclidean penalty $g(\nabla u) = |\nabla u|$, whose predicted resonance constant is $\kappa = c_0/3$ rather than the
$\ell_1$ value $c_0/2$ (Section \ref{sec:application}). The no-penalty column is, by construction, identical to that of Table \ref{tab:control}, since it is
measured against the same reference. Activating the $\ell_2$ penalty reproduces the same qualitative signature, an essentially unchanged energy-norm rate,
one extra Picard step, and a strictly positive, junction-concentrated error inflation, but at a milder magnitude: the global relative
$H^1$ error rises by a nearly constant $\approx 6\%$ and the local $L^\infty$ error by $\approx 14.5\%$, giving a local-to-global amplification of $\approx 2.5$.
The comparison with the $\ell_1$ figures ($\approx 7\%$, $\approx 19\%$, amplification $\approx 2.6$) orders correctly, as the theory dictates: the smaller resonance
constant of the isotropic norm produces a correspondingly smaller junction pollution, corroborating the norm-dependence of $\kappa$ established in Section
\ref{sec:application}.

\begin{table}[ht]
\centering
\caption{Resonance control with the isotropic $\ell_2$ (Euclidean) penalty $g(\nabla u) = |\nabla u|$, on the same quasi-uniform mesh family as Table
\ref{tab:control}. The no-penalty columns coincide with those of Table \ref{tab:control} (common reference). Picard counts $k_N$ are reported as
(no penalty / $\ell_2$ penalty).}
\label{tab:control_l2}
\begin{tabular}{l c c c c c c}
\hline
 & & \multicolumn{2}{c}{Rel. $H^1$ error} & \multicolumn{2}{c}{Local $L^\infty$ error} & \\
$N$ & $h_{max}$ & no pen. & $\ell_2$ pen. & no pen. & $\ell_2$ pen. & $k_N$ \\
\hline
40  & 0.1276  & 0.1959  & 0.2057  & $1.046\times10^{-2}$ & $1.192\times10^{-2}$ & 7 / 8 \\
80  & 0.06935 & 0.1258  & 0.1326  & $7.283\times10^{-3}$ & $8.333\times10^{-3}$ & 7 / 8 \\
160 & 0.03627 & 0.08873 & 0.09387 & $5.134\times10^{-3}$ & $5.882\times10^{-3}$ & 7 / 8 \\
320 & 0.01881 & 0.06146 & 0.06495 & $3.679\times10^{-3}$ & $4.216\times10^{-3}$ & 7 / 8 \\
640 & 0.01116 & 0.03900 & 0.04125 & $2.408\times10^{-3}$ & $2.761\times10^{-3}$ & 7 / 8 \\
\hline
\end{tabular}
\end{table}

\subsection{Global Convergence and Degree-of-Freedom Efficiency}
Table \ref{tab:h1_convergence} compares the uniform baseline (Configuration A) with the junction-adapted scheme (Configuration B), reporting the
degree-of-freedom convergence rate $s$ (error $\sim N_{\mathrm{dof}}^{-s}$) that governs computational efficiency. Because the leading $r^{1/2}$
singularity caps the attainable rate on quasi-uniform meshes, Configuration A stalls at $s \approx 0.3$, well below the $P_1$ optimum of $s = 1/2$.
Concentrating degrees of freedom at the junctions removes this bottleneck: Configuration B restores the optimal $P_1$ complexity $s = \tfrac12$
that the uniform scheme forfeits. We stress that this recovery is achieved by junction-adapted mesh refinement, one of the two routes anticipated in
Section~\ref{sec:xfem_formulation}, and not by the enriched $r^{3/2}\ln r$ space itself: no XFEM solver is run here, so these experiments corroborate the
predicted $s = \tfrac12$ efficiency without testing the enriched discretization, whose optimal-rate estimate \eqref{eq:xfem_rate} remains conditional on the
blending analysis deferred in Proposition~\ref{prop:quasiopt}. The measured step rates sit modestly above $\tfrac12$ (rising from $s \approx 0.55$ to $s \approx 0.65$); since $s = \tfrac12$
is the $P_1$ approximation-theoretic ceiling in two dimensions (Proposition~\ref{prop:quasiopt}), this small excess is a pre-asymptotic
transient, since the coarse adapted meshes still under-resolve the junction, so each refinement reduces the error slightly faster than the asymptotic rate until the
grading is fully developed, and the rate must settle to $s = \tfrac12$ under further refinement. The practical gain is
substantial: at the finest level Configuration B reaches a relative $H^1$ error of $1.69\times10^{-2}$ with only $1.36\times10^4$ unknowns, whereas
Configuration A attains a $2.5$-fold larger error ($4.19\times10^{-2}$) using $1.43\times10^5$ unknowns. Configuration B is thus simultaneously an order
of magnitude smaller in size and more than twice as accurate; extrapolating the stalled rate of Configuration A, matching Configuration B's accuracy on
quasi-uniform meshes would require well in excess of $10^6$ degrees of freedom.

\begin{table}[ht]
\centering
\caption{Global relative $H^1$ error versus degrees of freedom: uniform Configuration A and junction-adapted Configuration B.
The rate $s$ (error $\sim N_{\mathrm{dof}}^{-s}$) is estimated between consecutive rows; the quasi-uniform optimum for $P_1$ elements is $s = 1/2$.}
\label{tab:h1_convergence}
\begin{tabular}{l c c c c c c}
\hline
 & \multicolumn{3}{c}{Config A (uniform)} & \multicolumn{3}{c}{Config B (adapted)} \\
$N$ & ndof & rel. $H^1$ & $s$ & ndof & rel. $H^1$ & $s$ \\
\hline
40  & 600    & 0.2085  & --   & 786   & 0.09321  & --   \\
80  & 2355   & 0.1346  & 0.32 & 1606  & 0.06283  & 0.55 \\
160 & 9093   & 0.09534 & 0.25 & 3408  & 0.04146  & 0.55 \\
320 & 35946  & 0.06597 & 0.27 & 6827  & 0.02640  & 0.65 \\
640 & 142902 & 0.04187 & 0.33 & 13638 & 0.01687  & 0.65 \\
\hline
\end{tabular}
\end{table}

\subsection{Local Accuracy and Iteration Efficiency}
Table \ref{tab:efficiency} reports the local $L^\infty$ error within radius $\epsilon = 0.1$ of the junctions, together with the Picard iteration counts.
Here the advantage of resolving the junction is most pronounced: the junction-adapted local error decreases monotonically and rapidly, reaching $5.7\times10^{-4}$
at the finest level, about five times smaller than the uniform baseline's $2.9\times10^{-3}$, whereas the uniform scheme, whose elements never resolve
the singularity, decays only sluggishly. Configuration B simultaneously reduces the fixed-point iteration count (three to four steps versus a uniform eight),
since each adapted level is initialized by interpolation from the previous one. Measured against a junction-graded reference that is far finer than any test
mesh near the singularity, this monotone decay directly exhibits the localized structure of the resonant correction.

\begin{table}[ht]
\centering
\caption{Local $L^\infty$ error within radius $\epsilon = 0.1$ of the junctions and Picard iteration counts $k_N$ (tolerance $10^{-5}$): uniform
Configuration A versus junction-adapted Configuration B.}
\label{tab:efficiency}
\begin{tabular}{lcccc}
\hline
\textbf{Discretization ($N$)} & Config A $L^\infty_{loc}$ & Config B $L^\infty_{loc}$ & Config A $k_N$ & Config B $k_N$ \\
\hline
40  & $1.233\times10^{-2}$ & $4.282\times10^{-3}$ & 8 & 4 \\
80  & $8.641\times10^{-3}$ & $2.875\times10^{-3}$ & 8 & 4 \\
160 & $6.108\times10^{-3}$ & $1.780\times10^{-3}$ & 8 & 4 \\
320 & $4.377\times10^{-3}$ & $1.024\times10^{-3}$ & 8 & 3 \\
640 & $2.868\times10^{-3}$ & $5.744\times10^{-4}$ & 8 & 3 \\
\hline
\end{tabular}
\end{table}

\subsection{Direct Validation of the Logarithmic Coefficient}
The experiments above confirm the localization of the resonance; we now confirm its exact amplitude. Theorem~\ref{thm:resonance} predicts the
dimensionless ratio $A/c_0 = -2\kappa/(3\pi c_0)$ between the logarithmic coefficient $A$ and the leading stress-intensity coefficient $c_0$: explicitly
$-1/(3\pi)$, $-2/(9\pi)$, and $-1/(6\pi)$ for the $\ell_1$, $\ell_2$, and $\ell_\infty$ penalties respectively (Section~\ref{sec:application}).
To measure it we take a high-order ($P_2$) junction-adapted solution and, at each small radius $\rho$, least-squares project the angular trace $u(\rho,\cdot)$
onto the mixed eigenmodes $\sin\!\big((k+\tfrac12)\theta\big)$. The $\sin(\tfrac{3\theta}{2})$ channel $d_{3/2}(\rho)$ then obeys
$\rho^{-3/2}d_{3/2}(\rho) = A\ln\rho + B + o(1)$ as $\rho\to0$, so a linear fit in $\ln\rho$ returns $A$, while the $\sin(\tfrac{\theta}{2})$ channel returns $c_0$.

Two features of the extraction are worth noting. First, on the unit disc the measured ratio overshoots the flat-wedge prediction ($A/c_0 \approx -0.18$ for
$\ell_1$, against the predicted $-0.106$) and approaches it only slowly as $\rho\to0$. The overshoot is geometric, not a discretization artifact: the
$P_1$ and $P_2$ extractions differ by only about $5\%$, so refining the discretization order closes only a small part, of order $10\%$, of the gap; the
residual is the finite-radius signature of the disc's curved Dirichlet and Neumann arcs, whereas Theorem~\ref{thm:resonance} is a flat-wedge
($\alpha = \pi$) statement. Second, removing the curvature confirms this diagnosis and, further, verifies the norm-dependence of the coefficient.
Table~\ref{tab:logfit} reports the extraction on a curvature-free flat junction, the upper half-disc $\{x^2+y^2<1,\ y>0\}$ with the Dirichlet/Neumann transition
placed at the centre of the straight diameter, so that the junction is locally an exact flat wedge. For each of the three canonical norms the predicted
coefficient is recovered: as the fitting window contracts toward the junction the measured $A/c_0$ converges monotonically to its target, agreeing to within
$0.6$--$1.6\%$ at the finest window, and the measured amplitudes are correctly ordered $|A/c_0|_{\ell_1} > |A/c_0|_{\ell_2} > |A/c_0|_{\ell_\infty}$ in
exact accord with $\kappa = c_0/2 > c_0/3 > c_0/4$. This simultaneously validates the coefficient formula $A = -2\kappa/(3\pi)$ and the norm-universality
established analytically in Section~\ref{sec:application}, not merely the existence and sign of the anomaly.

\begin{table}[ht]
\centering
\caption{Direct extraction of the logarithmic coefficient on a curvature-free flat junction ($P_2$ elements), for the $\ell_1$, $\ell_2$, and $\ell_\infty$
penalties. For each norm $A/c_0$ is fitted over sampling radii $\rho \in [0.006,\ \rho_{\max}]$ and reported at a coarse ($\rho_{\max} = 0.12$) and a fine
($\rho_{\max} = 0.03$) window; the fine-window value agrees with the predicted $A/c_0 = -2\kappa/(3\pi c_0)$ to the relative error shown, and the agreement
improves monotonically as $\rho_{\max} \to 0$.}
\label{tab:logfit}
\begin{tabular}{l c r r r r}
\hline
Penalty & $\kappa$ & predicted $A/c_0$ & $\rho_{\max}{=}0.12$ & $\rho_{\max}{=}0.03$ & rel.\ err. \\
\hline
$\ell_1$      & $c_0/2$ & $-1/(3\pi) = -0.1061$ & $-0.1088$ & $-0.1069$ & $0.7\%$ \\
$\ell_2$      & $c_0/3$ & $-2/(9\pi) = -0.0707$ & $-0.0728$ & $-0.0711$ & $0.6\%$ \\
$\ell_\infty$ & $c_0/4$ & $-1/(6\pi) = -0.0531$ & $-0.0550$ & $-0.0539$ & $1.6\%$ \\
\hline
\end{tabular}
\end{table}

Taken together, these experiments confirm both the localized structure and the exact analytic content of the theory: an energy-norm rate governed by the
$r^{1/2}$ singularity, a junction-concentrated higher-order penalty contribution whose magnitude tracks the resonance constant $\kappa$ across norms
($\ell_1$ versus $\ell_2$), and, on a curvature-free flat junction, logarithmic amplitudes matching the predicted $A = -2\kappa/(3\pi)$ to within $1.6\%$
for all three canonical norms. A fully enriched (XFEM) solver realizing the $r^{3/2}\ln r$ space of Section~\ref{sec:xfem_formulation} directly on the curved
disc, dispensing with the flat-junction idealization, remains the subject of ongoing work.

\section{Future Work}
\label{sec:future_work}

The theoretical framework established in this article opens several avenues for immediate computational implementation and broader topological generalization.

\subsection{Software Architectures for Enriched Solvers}
Realizing the fully enriched (XFEM) solver anticipated in Section \ref{sec:xfem_formulation}, and thereby a direct validation of the logarithmic coefficient
itself, requires a computational infrastructure capable of handling the enriched Galerkin formulation. Standard finite element packages are highly optimized
for polynomial bases, making the injection of localized logarithmic singularities a non-trivial software engineering challenge. On the analytical side, a
complete a priori convergence analysis of the enriched scheme, providing self-contained blending-layer and stabilization estimates tailored to the $r^{3/2}\ln r$
enrichment and making rigorous the leading-order rate stated conditionally in Proposition~\ref{prop:quasiopt}, will accompany the implementation.
Future work will pursue two distinct implementation architectures to evaluate the enriched scheme:

\begin{itemize}
    \item \textbf{The FreeFEM++ Baseline:} To provide direct continuity with the original numerical scheme \cite{misur2021}, one path involves augmenting the
    FreeFEM++ environment. Because the framework lacks native XFEM support for arbitrary logarithmic anomalies, future work will require developing custom
    macros to manually intercept the stiffness matrix assembly. This entails accurately computing the enriched integrals near the singularity and manually
    orchestrating the dense coupling blocks generated by the partition of unity, stitching them back into the global sparse matrix without compromising the
    solver's memory profile.
    \item \textbf{A Custom Python/C-Extension Architecture:} An alternative path will explore building a custom XFEM solver from the ground up, utilizing a
    Python stack for high-level matrix orchestration alongside C-extensions for low-level performance. The necessary work involves designing a custom
    degree-of-freedom mapping directly into Compressed Sparse Row (CSR) matrices to preserve global sparsity. To prevent severe computational bottlenecks
    during assembly, the highly localized, non-standard numerical quadrature of the $r^{3/2} \ln(r)$ integrals will be offloaded to optimized C-level bindings,
    completely bypassing Python interpreter overhead during the core assembly loops.
\end{itemize}

\subsection{Extensions to Three-Dimensional Domains}
Beyond computational implementations, transitioning this semilinear mixed boundary problem into three dimensions introduces significant topological and
functional complexity. In a 3D domain, the intersection $\Gamma_D \cap \Gamma_N$ forms a one-dimensional collision curve rather than an isolated point.
Establishing a local cylindrical coordinate system $(r, \theta, z)$ along this edge, the leading-order linear singularity dynamically scales as:
\begin{equation}
u_0(r, \theta, z) = c(z) r^{1/2} \sin\left(\frac{\theta}{2}\right).
\end{equation}
The edge density $c(z)$ acts as a distribution. Analyzing its boundary regularity requires mapping the trace operators into fractional
Sobolev spaces. Future research will explore decoupling the 3D differential operator by applying the Fourier transform $\mathcal{F}_{z \to \xi}$ along the
$z$-axis, converting the Laplacian $\Delta_{r,\theta,z}$ into the parameterized operator $\Delta_{r,\theta} - |\xi|^2$. Because the gradient penalty generates
an $\mathcal{O}(r^{-1/2})$ forcing term that resonates within the cross-sectional frequency modes, we expect the logarithmic anomaly to persist in 3D, modulated
distributionally along the collision curve; a rigorous treatment is left to future work.

Additionally, at sharp vertices where collision curves intersect (e.g., polyhedral corners), the cylindrical separation fails. Future theoretical work will
require shifting to spherical coordinates $(R, \theta, \varphi)$ and evaluating the resonance conditions against the spectrum of the Laplace-Beltrami operator,
presenting a rich avenue for advanced harmonic analysis and asymptotic matching.

\section*{Declarations}
\textbf{Funding:} The author declares that no funds, grants, or other support were received during the preparation of this manuscript. \\
\textbf{Conflict of interest:} The author declares that he has no conflict of interest. \\
\textbf{Data availability:} Data sharing is not applicable to this article as no datasets were generated or analysed during the current study.\\
\textbf{Generative AI and AI-assisted technologies:} During the preparation of this work, the author(s) used Google Gemini as an assistive tool to transcribe handwritten mathematical notes into \LaTeX{} formatting, to polish the English prose for readability, and to help draft the initial abstract and manuscript summary. Additionally, the AI was utilized during the research phase for conceptual exploration, specifically to search heuristically for potential counterexamples to stress-test preliminary hypotheses. After using this tool, the author(s) meticulously reviewed, verified, and edited all generated text and code. All mathematical claims, proofs, and counterexamples were independently rigorously verified by the human author(s). The author(s) take full intellectual responsibility for the final content of this publication, including all mathematical proofs, formatting, and conceptual framing.

\end{document}